\documentclass[11pt, reqno]{amsart}
\usepackage[utf8]{inputenc}
\usepackage{amsmath, amssymb, amsfonts, latexsym, amsthm}
\usepackage{calligra}
\usepackage{tcolorbox}

\newtheorem{thm}{Theorem}[section]
\newtheorem{cor}[thm]{Corollary}
\newtheorem{lem}[thm]{Lemma}
\newtheorem{prop}[thm]{Proposition}
\newtheorem{rem}[thm]{Remark}
\newtheorem{term}[thm]{Terminology}

\newtheorem{defn}[thm]{Definition}
\numberwithin{equation}{section}

\newcommand{\z}{\zeta}

\newcommand{\ov}{\overline}

\newcommand{\no}{\noindent}

\newcommand{\Om}{\Omega}

\newcommand{\abs}[1]{\left\vert#1\right\vert}

\begin{document}

\title{A Canonical Positive Definite Kernel Associated with the $\xi$-Bergman Kernel}
 
\author{Sahil Gehlawat, Aakanksha Jain and Amar Deep Sarkar} 

\address{SG: Department of Mathematics, Indian Institute of Technology Jodhpur, NH 65, Jodhpur, Rajasthan-342030, India}
\email{sahilg@iitj.ac.in}

\address{AJ: Department of Mathematics, Indian Institute of Science Education and Research, Pune 411008, India}
\email{aakanksha.jain@iiserpune.ac.in, mathcv.jain@gmail.com}

\address{ADS: Department of Mathematics, Indian Institute of Technology Bhubaneswar, Argul 752050, India}
\email{amar@iitbbs.ac.in}

\begin{abstract}
Let $\Om \subset \mathbb{C}^{n}$ and $\xi \in \ell^{1}$. The $\xi$-Bergman kernel $K_{\xi, \Omega}$, introduced by Bao and Guan, generalizes the classical Bergman kernel by replacing the point evaluation functional with a functional determined by sequence $\xi$. While this kernel inherits several important extremal and plurisubharmonic properties, it is intrinsically an on-diagonal object and therefore lacks the two-variable reproducing kernel structure that lies at the heart of the classical Bergman theory. The purpose of this paper is to associate a canonical Hermitian positive-definite kernel with the $\xi$-Bergman kernel and to investigate its analytic and geometric properties.

Our construction is based on the family of Riesz representatives corresponding to the $\xi$-evaluation functionals. More precisely, we introduce a Hermitian kernel obtained as the Gram kernel of these representatives and show that it is positive definite and for $z \in \Om$ satisfies 
\[
B_{\xi,\Omega}(z,z)=K_{\xi,\Omega}(z),
\]
thereby recovering the $\xi$-Bergman kernel as its diagonal restriction. As a consequence, we prove that the $\xi$-Bergman kernel is real analytic on $\Omega$. We also establish biholomorphic transformation laws, and obtain a representation of the $\xi$-Bergman kernel in terms of derivatives of the classical Bergman kernel. 

Furthermore, we obtain explicit formulas for the $\xi$-Bergman kernel on the upper half-plane $\mathbb{H}$ corresponding to several classes of sequences $\xi$, establish corresponding $\xi$-Lu Qi-Keng results, and derive precise boundary asymptotics. These examples illustrate how the choice of the differential functional influences both the zero set and the boundary growth of the associated kernel.
\end{abstract}

\maketitle

\section{Introduction}

The Bergman kernel of a domain $\Om \subset \mathbb{C}^{n}$, denoted by $K_{\Om}(\cdot, w)$, is the reproducing kernel of $A^{2}(\Om)$ corresponding to the evaluation functional $\Lambda(f) = f(w)$ for $w \in \Om$. It occupies a central position in complex analysis and complex geometry. In addition to providing the reproducing kernel of the Bergman space, it encodes a remarkable amount of analytic and geometric information about the underlying domain. It gives rise to the Bergman metric, determines biholomorphic invariants, and plays an essential role in the study of invariant metrics, holomorphic mappings, and several problems in complex geometry. Owing to these diverse applications, numerous generalizations of the classical Bergman kernel have been developed over the past decades, each arising from a different extremal problem or geometric consideration.

\no Among these generalizations is the $\xi$-Bergman kernel introduced by Bao and Guan (see \cite{Bao-Guan_1}), where the classical point evaluation functional is replaced by a differential functional
\[
L(f)(z) = (\xi \cdot f)(z) = \sum_{\alpha \in \mathbb{N}^{n}}{\xi_{\alpha}\frac{f^{(\alpha)}(z)}{\alpha!}},
\]
determined by a sequence $\xi = \{\xi_{\alpha}\}_{\alpha \in \mathbb{N}^{n}}$. The $\xi$-Bergman kernel (see Definition \ref{md1}) is defined as $K_{\xi, \Om}(z) := \text{sup}_{0 \neq f \in A^{2}(\Omega)}{\frac{\abs{(\xi \cdot f)(z)}^{2}}{\int_{\Omega}{\abs{f}}^{2}}}$, for all $z \in \Omega$. This construction naturally interpolates between point evaluation and higher-order derivative evaluations and has proven to be a useful tool in several problems related to the strong openness property of multiplier ideal sheaves. The theory was further developed in subsequent works (see \cite{Bao-Guan_2}, \cite{BGY2}, \cite{BG3} and \cite{Bao_note}). More recently, in \cite{Bao-Guan-Sun}, Bao, Guan and Sun developed an $L^p$-theory of $\xi$-Bergman kernels, extending the classical Hilbert-space setting to nonlinear extremal problems, using similar techniques as in the work of Zhang and Chen (see \cite{p-Bergman}) . Their work establishes reproducing identities, regularity properties, and several structural results for the associated $p$-Bergman kernels corresponding to a sequence $\xi$.

\no Although these developments provide a comprehensive understanding of the extremal aspects of $\xi$-Bergman kernels, they also reveal a fundamental distinction from the classical Bergman theory. The classical Bergman kernel is simultaneously an extremal object, a reproducing kernel, and the Gram kernel of the family of point evaluation functionals. Consequently, it carries the complete Hilbert-space geometry of the Bergman space. In contrast, the $\xi$-Bergman kernel is naturally defined only on the diagonal. Although the recent $L^p$-theory introduces an off-diagonal extremal kernel, it does not produce a positive-definite reproducing kernel whose diagonal coincides with the original $\xi$-Bergman kernel. Thus, unlike the classical situation, there is no canonical two-variable kernel naturally associated with the $\xi$-Bergman kernel itself.

The purpose of this article is to fill this gap. We introduce a Hermitian positive-definite kernel (see Definition \ref{D: The xi-Bergman kernel B}) $B_{\xi, \Omega}(\cdot, \cdot)$ such that for all $z \in \Om$
\begin{equation}\label{E: the diagonal restriction of xi Bergman kernel}
B_{\xi,\Omega}(z,z) = K_{\xi,\Omega}(z),
\end{equation}
that is, the $\xi$-Bergman kernel is realized as the diagonal restriction of the kernel $B_{\xi, \Omega}(\cdot, \cdot)$. This restores an essential feature of the classical Bergman theory that was previously unavailable for the $\xi$-Bergman kernel. The construction of $B_{\xi,\Omega}$ is canonical. Once the differential functional that defines the $\xi$-Bergman kernel is fixed, the associated family of Riesz representatives is uniquely determined by the Hilbert-space structure of the Bergman space. The kernel $B_{\xi,\Omega}$ therefore arises naturally as the Gram kernel of this family and does not require auxiliary choices. From this viewpoint, it may be regarded as the natural two-variable object associated with the $\xi$-Bergman kernel in much the same way that the classical Bergman kernel is associated with point evaluation.


Our first objective is to establish the fundamental analytic and geometric properties of $B_{\xi,\Omega}$. We show that it admits an explicit representation in terms of derivatives of the classical Bergman kernel, from which several basic properties follow naturally. In particular, we prove biholomorphic transformation laws for both the kernel $B_{\xi, \Omega}$ and the associated $\xi$-Bergman kernel $K_{\xi, \Omega}$. The positive-definite nature of $B_{\xi,\Omega}$ also gives rise to a natural reproducing kernel Hilbert space. Although the original Bergman space already provides the ambient Hilbert space in which the $\xi$-evaluation functionals are defined, the kernel $B_{\xi, \Omega}$ determines another holomorphic reproducing kernel Hilbert space whose reproducing kernel is precisely $B_{\xi,\Omega}$. This additional Hilbert-space structure offers a new perspective on $\xi$-Bergman kernels and further emphasizes the role of the kernel $B_{\xi, \Omega}$ as the natural two-variable object associated with the differential functional $\xi$.

\no A particularly useful feature of the representation formula for $B_{\xi,\Omega}$ is that it reduces many questions to the corresponding questions about the classical Bergman kernel. Consequently, analytic properties of the latter may often be transferred directly to the former. As an immediate application, we prove that the $\xi$-Bergman kernel is real analytic on $\Om$ (see Theorem \ref{T: Real analyticity of xi Bergman kernel K}). Previously, it was only known to be locally Lipschitz continuous function (Proposition 2.9 in \cite{Bao-Guan-Sun}).

The second part of the paper is devoted to explicit computations on the upper half-plane $\mathbb{H} \subset \mathbb{C}$. The upper half-plane serves as an ideal model because its classical Bergman kernel is known explicitly, allowing one to obtain closed formulas for the kernel $B_{\xi, \mathbb{H}}(\cdot, \cdot)$ corresponding to various choices of the sequence $\xi$. We first consider finitely supported sequences and derive explicit expressions for the associated kernels. We then study exponentially decaying sequences, in particular $\xi_n=\frac{1}{n!}$,
for which the kernel $B_{\xi, \mathbb{H}}(\cdot, \cdot)$ admits a representation involving classical special functions known as Bessel functions. These explicit formulas provide concrete examples illustrating the general theory and reveal several phenomena that are not visible from the abstract construction alone.

\no One consequence of these formulas is the possibility of studying the zero sets of the $\xi$-Bergman kernel. Motivated by the classical Lu Qi-Keng problem, we investigate whether the $\xi$-Bergman kernel vanishes on $\mathbb{H}$. For the classes of sequences considered in this article, we establish corresponding $\xi$-Lu Qi-Keng theorems by combining explicit kernel formulas with classical results on the zeros of entire functions and Bessel functions. These examples illustrate how the choice of the sequence $\xi$ influences the global analytic behaviour of the associated kernel.

\no Finally, we investigate the boundary behaviour of the $\xi$-Bergman kernel. Using the explicit formulas obtained earlier, we derive precise asymptotic expansions as one approaches the boundary of $\mathbb{H}$. These asymptotics reveal two markedly different regimes. For finitely supported sequences, the $\xi$-Bergman kernel exhibits polynomial growth comparable to that of higher-order derivatives of the classical Bergman kernel. In contrast, for exponentially decaying sequences such as $\xi_n=1/n!$, the boundary behaviour is governed by exponential growth. This dichotomy illustrates that the asymptotic behaviour of the $\xi$-Bergman kernel is highly sensitive to the analytic properties of the sequence $\xi$, thereby providing a new family of boundary phenomena in the theory of generalized Bergman kernels.

Taken together, these results show that the canonical kernel $B_{\xi, \Omega}$ introduced in this paper is considerably more than an auxiliary two-variable extension of the $\xi$-Bergman kernel. It provides a natural Hilbert-space framework in which the original $\xi$-Bergman kernel appears as the diagonal of a positive-definite Hermitian kernel, establishes a direct connection with the classical Bergman kernel through explicit representation formulas, and offers a unified approach to problems concerning analyticity, biholomorphic invariance, zero sets, and boundary asymptotics.

The present work should be viewed as complementary to the recent $L^p$-theory of Bao, Guan and Sun (\cite{Bao-Guan-Sun}). Their work develops the extremal theory of $\xi$-Bergman kernels in the general nonlinear setting, establishing existence, uniqueness, regularity and several structural properties of the associated $p$-Bergman kernels. In contrast, the present paper exploits the additional Hilbert-space structure available in the case $p=2$ to construct a canonical positive-definite kernel naturally associated with the $\xi$-Bergman kernel. While the off-diagonal kernels considered in the $L^p$-theory arise from extremal problems, our construction is obtained as the Gram kernel of the family of Riesz representatives corresponding to the $\xi$-evaluation functionals. This viewpoint produces a canonical Hermitian kernel whose diagonal is exactly the original $\xi$-Bergman kernel and leads naturally to a reproducing kernel Hilbert space, explicit representation formulas in terms of derivatives of the classical Bergman kernel, and the applications developed in the later sections.

The paper is organized as follows. In Section~2, we introduce the kernel $B_{\xi,\Omega}$ and establish its fundamental properties, including its positivity, Hermitian symmetry, real analyticity, and the reproducing kernel Hilbert space naturally associated with it. In Section~3, we relate this kernel to the $\xi$-Bergman kernel by identifying the function $R_{\xi,\Omega}$ with the off-diagonal $\xi$-Bergman kernel and proving that
\[
B_{\xi,\Omega}(z,z)=K_{\xi,\Omega}(z).
\]
We also obtain a representation of these kernels in terms of derivatives of the classical Bergman kernel and deduce the real analyticity of the $\xi$-Bergman kernel. Section~4 is devoted to explicit computations on the upper half-plane, where we derive closed formulas for several classes of sequences $\xi$ and investigate the corresponding $\xi$-Lu Qi-Keng problem. In Section~5, we study the boundary behaviour of the $\xi$-Bergman kernel on the upper half-plane and obtain precise asymptotic estimates for several classes of sequences $\xi$. Finally, in Section~6, we establish biholomorphic transformation laws for the $\xi$-Bergman kernel and the kernel $B_{\xi, \Omega}$.
\medskip

\noindent \textbf{Notation:} Due to property \ref{E: the diagonal restriction of xi Bergman kernel}, we will also call the kernel $B_{\xi, \Omega}$ the $\xi$-Bergman kernel, whenever there is no sense of confusion. To distinguish it from the kernel introduced by Bao-Guan-Sun in \cite{Bao-Guan-Sun}, we will call their kernel $K_{\xi, \Omega}(\cdot, \cdot)$ the off-diagonal $\xi$-Bergman kernel, just as they did.
Throughout the article, we shall drop the subscript $\Om$ whenever there is no confusion about the domain.

\section{The canonical positive-definite kernel $B_{\xi}$}

Let $\Omega \subset \mathbb{C}^n$ be a domain. The \textit{Bergman space} $A^2(\Om)$ is the reproducing kernel Hilbert space of square-integrable holomorphic functions on $\Om$ with respect to the inner product
\[
\langle f, g\rangle_{A^2(\Om)} :=
\int_{\Om} f(z) \overline{g(z)}\,dV(z),
\]
where $dV$ denotes the standard Lebesgue measure. 
The reproducing kernel $K_{\Omega}(z,\zeta)$ of $A^2(\Om)$ is called the \textit{Bergman kernel} of $\Omega$. Thus, we have the reproducing formula
\[
f(z) = \langle f, K_{\Om}(\cdot,z)\rangle_{A^2(\Om)},
\qquad
\text{for every }f\in A^2(\Om) \text{ and }z\in\Om.
\]
It is known that $K_{\Om}(z,\zeta)$ is holomorphic in $z$ and antiholomorphic in $\zeta$. For multi-indices $\alpha,\beta \in \mathbb{N}^n$, we set
\[
K_{\Omega}^{\alpha,\ov\beta}(z,\zeta)
:=
\frac{\partial^{|\alpha|+|\beta|}K_{\Omega}(z,\zeta)}
{\partial z^\alpha \partial \overline{\zeta}^{\beta}}.
\]
The following reproducing identity for derivatives of the Bergman kernel is standard and follows by differentiating the Bergman kernel reproducing formula; see, for instance, \cite{Bell, KrantzBergman, KrantzSCV}.

\begin{lem}[Derivative Reproducing Formula]\label{der}
For every $f\in A^2(\Omega)$, $z\in\Omega$ and $\alpha\in\mathbb N^n$,
\[
f^{(\alpha)}(z)
=
\langle
f,K_{\Omega}^{0,\overline{\alpha}}(\cdot,z)
\rangle_{A^2(\Omega)}.
\]
Consequently,
\[
K_{\Omega}^{\alpha,\overline{\beta}}(z,\zeta)
=
\left\langle
K_{\Omega}^{0,\overline{\beta}}(\cdot,\zeta),
K_{\Omega}^{0,\overline{\alpha}}(\cdot,z)
\right\rangle_{A^2(\Omega)}
\]
for all $z,\zeta\in\Omega$ and $\alpha,\beta\in\mathbb N^n$.
\end{lem}

\medskip

Consider the space
\[
\ell^1=
\left\{
\xi=(\xi_\alpha)_{\alpha\in\mathbb N^n}:
\sum_{\alpha\in\mathbb N^n}
|\xi_\alpha|\,\rho^{|\alpha|}
<\infty
\text{ for every }\rho>0
\right\}.
\]
Given $\xi\in\ell^1$,
we first introduce the kernel function $R_\xi:\Omega\times\Omega\rightarrow\mathbb C$,
from which we shall construct the positive definite-kernel $B_\xi$. In the next section, we shall see that $R_\xi$ is precisely the off-diagonal $\xi$-Bergman kernel; see Definition \ref{md2}.

\begin{defn}
Let $\Om\subset\mathbb{C}^n$ and $\xi\in \ell^1$. We define the kernel function $R_{\xi}$ by
\begin{equation}
R_\xi(z,\zeta)
=
\sum_{\alpha\in\mathbb N^n}
\frac{\ov{\xi_\alpha}}{\alpha!}
K_{\Omega}^{0,\ov\alpha}(z,\zeta).
\end{equation}
for $z,\zeta\in\Om$.
\end{defn}

\begin{prop}\label{rxi_0}
The series 
\[
\sum_{\alpha\in\mathbb N^n}
\frac{\ov{\xi_\alpha}}{\alpha!}
K_{\Omega}^{0,\ov\alpha}(z,\zeta),
\qquad
\text{for }z,\zeta\in\Om
\]
converges uniformly on all compact subsets of $\Om\times \Om$. Therefore, $R_{\xi}(z,\zeta)$ is holomorphic in $z$ and antiholomorphic in $\zeta$. In particular, $R_{\xi}(z,\zeta)$ is real-analytic on $\Om\times\Om$.
\end{prop}

\begin{proof}
Let $K\subset \Om$ be a compact. Choose $r>0$ such that $\overline{P(p,r)}$ is compactly contained in $\Om$ for every $p\in K$, where $P(p,r)$ denotes the polydisc centered at $p$ with polyradius $r$.
There exists a constant $A_K>0$ such that 
\[
\vert K_{\Om}(z,\zeta)\vert \leq A_K
\]
for $z,\zeta$ lying in any of such polydiscs. Therefore, by the Cauchy estimates
\[
\vert K_{\Omega}^{0,\ov\alpha}(z,\zeta)\vert 
\leq 
A_K \frac{\alpha !}{r^{\vert \alpha\vert}}
\]
for all $z,\zeta \in K$ and $\alpha\in \mathbb{N}^n$. Hence
\[
\left|
\frac{\ov{\xi_\alpha}}{\alpha!}
K_\Omega^{0,\ov \alpha}(z,\zeta)
\right|
\leq
A_K\,|\xi_\alpha|\,r^{-|\alpha|}.
\]
Since $\xi\in \ell^1$, we have 
$\sum_{\alpha\in\mathbb N^n}
|\xi_\alpha|r^{-|\alpha|}
<\infty.
$
Therefore, by the Weierstrass $M$-test, the series
\[
\sum_{\alpha\in\mathbb N^n}
\frac{\overline{\xi_\alpha}}{\alpha!}
K_\Omega^{0,\ov\alpha}(z,\zeta)
\]
converges uniformly on $K\times K$. Since $K\subset\Omega$ was arbitrary, the convergence is locally uniform on $\Omega\times\Omega$. The other conclusions follow because $K_{\Om}(z,\zeta)$ is holomorphic in $z$ and antiholomorphic in $\zeta$.
\end{proof}

\begin{prop}\label{rxi}
For every $\zeta\in\Om$, we have
\[
R_{\xi}(\cdot,\zeta) \in A^2(\Om).
\]
\end{prop}

\begin{proof}
For $N\geq 0$, define the index set
\[
I_N:=\{\alpha\in\mathbb N^n:|\alpha|\leq N\}
\]
and define the partial sums
\[
M_N(z)
:=
\sum_{\alpha\in I_N}
\frac{|\xi_\alpha|}{\alpha!}
\left|K_\Omega^{0,\ov\alpha}(z,\zeta)\right|.
\]
We shall prove that the functions $M_N$ are uniformly bounded in
$L^2(\Omega)$. By expanding ${M_N}^2$, we get
\[
\int_\Omega M_N(z)^2\,dV(z)
=
\sum_{\alpha,\beta\in I_N}
\frac{|\xi_\alpha||\xi_\beta|}{\alpha!\beta!}
\int_\Omega
\left|K_\Omega^{0,\ov\alpha}(z,\zeta)\right|
\left|K_\Omega^{0,\ov\beta}(z,\zeta)\right|
\,dV(z).
\]
By Cauchy--Schwarz,
\[
\int_\Omega
\left|K_\Omega^{0,\ov\alpha}(z,\zeta)\right|
\left|K_\Omega^{0,\ov\beta}(z,\zeta)\right|
\,dV(z)
\leq
\left\|K_\Omega^{0,\ov\alpha}(\cdot,\zeta)\right\|_{A^2(\Om)}
\left\|K_\Omega^{0,\ov\beta}(\cdot,\zeta)\right\|_{A^2(\Om)}.
\]
Using the reproducing identity for the derivatives of the Bergman kernel,
\[
\left\|K_\Omega^{0,\ov\alpha}(\cdot,\zeta)\right\|_{A^2(\Om)}^2
=
K_\Omega^{\alpha,\ov\alpha}(\zeta,\zeta).
\]
Therefore,
\[
\int_\Omega M_N(z)^2\,dV(z)
\leq
\left(
\sum_{\alpha\in I_N}
\frac{|\xi_\alpha|}{\alpha!}
\left(K_\Omega^{\alpha,\ov\alpha}(\zeta,\zeta)\right)^{1/2}
\right)^2.
\]
Now choose $r_{\zeta}>0$ such that the closed polydisc $\overline{P(\zeta,r)}$ is compactly contained in $\Omega$. Then, there exists $A_{\zeta}>0$ such that
\[
|K_\Omega(w,\eta)|\leq A_{\zeta}
\]
for all $w,\eta\in \overline{P(\zeta,r_{\zeta})}$. By the several-variable Cauchy estimate,
\[
\left|K_\Omega^{\alpha,\ov\alpha}(\zeta,\zeta)\right|
\leq
A_{\zeta}\frac{(\alpha!)^2}{(r_{\zeta})^{2|\alpha|}}.
\]
Substituting this estimate gives
\[
\int_\Omega M_N(z)^2\,dV(z)
\leq
A_{\zeta}
\left(
\sum_{\alpha\in I_N}
|\xi_\alpha|(r_{\zeta})^{-|\alpha|}
\right)^2
\leq
A_{\zeta}
\left(
\sum_{\alpha\in \mathbb{N}^n}
|\xi_\alpha|(r_{\zeta})^{-|\alpha|}
\right)^2
< \infty,
\]
because $\xi\in \ell^1$.
Thus, there is a constant $C=C(\zeta,\xi)>0$, independent of $N$, such that
\[
\int_\Omega M_N(z)^2\,dV(z)\leq C.
\]
Since $M_N(z)$ increases pointwise to
\[
M(z):=
\sum_{\alpha\in\mathbb N^n}
\frac{|\xi_\alpha|}{\alpha!}
\left|K_\Omega^{0,\ov\alpha}(z,\zeta)\right|,
\]
the monotone convergence theorem gives
\[
\int_\Omega M(z)^2\,dV(z)
=
\lim_{N\to\infty}
\int_\Omega M_N(z)^2\,dV(z)
\leq C<\infty.
\]
Hence $M\in L^2(\Omega)$.
Finally,
\[
\left|
\sum_{\alpha\in\mathbb N^n}
\frac{\ov{\xi_\alpha}}{\alpha!}
K_\Omega^{0,\ov\alpha}(z,\zeta)
\right|
\leq
M(z),
\]
and $M\in L^2(\Omega)$. Therefore, $R_{\xi}(\cdot,\zeta)\in L^2(\Om)$, which proves the proposition.
\end{proof}

\begin{rem}\label{rem_rxi}
Varying $\zeta$ over a fixed compact set $K\subset\Omega$ in the
proof of Proposition \ref{rxi}, the constants $r_{\zeta}$ and $A_{\zeta}$ may be
chosen independently of $\zeta\in K$. The resulting estimates will show that the sequence of partial sums
\[
R_{\xi,N}(\cdot,\zeta): = \sum_{\alpha\in I_{N}}
\frac{\ov{\xi_\alpha}}{\alpha!}
K_{\Omega}^{0,\ov\alpha}(\cdot,\zeta),
\qquad
N\in\mathbb{Z}^+
\]
converges to
$R_{\xi}(\cdot,\zeta)$ in $A^2(\Omega)$, uniformly for $\zeta$ in
compact subsets of $\Omega$.
\end{rem}

We are now in a position to introduce the central object, the positive-definite kernel $B_{\xi}$, of this paper and establish its basic analytic and geometric properties.

\begin{defn}\label{D: The xi-Bergman kernel B}
Let $\Om\subset\mathbb{C}^n$ be a domain and $\xi\in \ell^1$. We define the kernel $B_{\xi}$ associated to $\Om$ by
    \begin{equation}
    B_\xi(z,\zeta)
    :=
    \langle R_\xi(\cdot,\zeta),R_\xi(\cdot,z)\rangle_{A^2(\Omega)}.
    \end{equation}
for $z,\zeta\in\Om.$
\end{defn}

\begin{thm}\label{sxi}
The kernel $B_{\xi}$ has the series representation
    \begin{equation}
    B_\xi(z,\zeta)
    =
    \sum_{\alpha,\beta\in\mathbb{N}^n}
    \frac{\xi_\alpha\overline{\xi_\beta}}{\alpha!\beta!}
    K_\Omega^{\alpha,\ov \beta}(z,\zeta),
    \end{equation}
where the series converges uniformly on all compact subsets of $\Omega\times\Omega$. The kernel $B_{\xi}$ satisfies the following properties :

\begin{enumerate}
    \item $B_\xi(z,\zeta)$ is holomorphic in $z$ and anti-holomorphic in $\zeta$. Hence, $B_{\xi}(z,\zeta)$ is real-analytic on $\Om\times\Om$.
    
    \item $B_\xi$ is Hermitian-symmetric:
    \[
    B_\xi(z,\zeta)
    =
    \overline{B_\xi(\zeta,z)}.
    \]

    \item $B_{\xi}(z, z) \geq 0$ for all $z \in \Omega$.

    \item The Cauchy--Schwarz estimate holds:
    \[
    |B_\xi(z,\zeta)|^2
    \leq
    B_\xi(z,z)B_\xi(\zeta,\zeta).
    \]

    \item $\log B_{\xi}(z, z)$ is a pluri-subharmonic function on $\Omega$. 

    \item $B_\xi$ is positive-definite. That is, for any $z_1,\dots,z_N\in\Omega$ and $c_1,\dots,c_N\in\mathbb C$,
    \[
    \sum_{i,j=1}^{N}
    c_i\overline{c_j}
    B_\xi(z_i,z_j)
    \geq 0.
    \]
\end{enumerate}
\end{thm}

\begin{proof}
For $N\geq 0$, consider the sequence of partial sums
\[
R_{\xi,N}(\cdot,\zeta)
=
\sum_{\alpha\in I_N}
\frac{\ov{\xi_\alpha}}{\alpha!}
K_\Omega^{0,\overline{\alpha}}(\cdot,\zeta).
\]
It follows from Remark \ref{rem_rxi} that 
\[
\Vert R_{\xi,N}(\cdot,\zeta) - R_\xi(\cdot,\zeta)\Vert_{A^2(\Om)} \rightarrow 0
\qquad\text{for every }\zeta\in\Om.
\]
Therefore, for fixed $z,\zeta\in\Omega$,
\begin{multline*}
\vert \langle R_{\xi,N}(\cdot,\zeta),R_{\xi,N}(\cdot,z)\rangle
-
\langle R_\xi(\cdot,\zeta),R_\xi(\cdot,z)\rangle\vert
\\
\leq
\vert \langle R_{\xi,N}(\cdot,\zeta)-R_{\xi}(\cdot,\zeta),R_{\xi,N}(\cdot,z)\rangle\vert
+
\vert \langle R_{\xi}(\cdot,\zeta), R_{\xi,N}(\cdot,z) - R_{\xi}(\cdot,z)\rangle\vert 
\longrightarrow
0
\end{multline*}
as $N\rightarrow\infty$. Now,
\begin{eqnarray*}
\langle R_{\xi,N}(\cdot,\zeta),R_{\xi,N}(\cdot,z)\rangle
&=&
\left\langle
\sum_{\beta\in I_N}
\frac{\ov{\xi_\beta}}{\beta!}
K_\Omega^{0,\overline{\beta}}(\cdot,\zeta),
\sum_{\alpha\in I_N}
\frac{\ov{\xi_\alpha}}{\alpha!}
K_\Omega^{0,\overline{\alpha}}(\cdot,z)
\right\rangle  \\
&=&
\sum_{\alpha,\beta\in I_N}
\frac{\xi_\alpha\overline{\xi_\beta}}{\alpha!\beta!}
\left\langle
K_\Omega^{0,\overline{\beta}}(\cdot,\zeta),
K_\Omega^{0,\overline{\alpha}}(\cdot,z)
\right\rangle 
=
\sum_{\alpha,\beta\in I_N}
\frac{\xi_\alpha\overline{\xi_\beta}}{\alpha!\beta!}
K_\Omega^{\alpha,\ov\beta}(z,\zeta).
\end{eqnarray*}
Letting $N\to\infty$, we obtain
\[
B_\xi(z,\zeta)
=
\sum_{\alpha,\beta\in\mathbb N^n}
\frac{\xi_\alpha\overline{\xi_\beta}}{\alpha!\beta!}
K_\Omega^{\alpha,\ov\beta}(z,\zeta).
\]
It remains to prove local uniform convergence. Let $K\subset\Omega$ be
compact. Choose $r>0$ such that, for every $p\in K$, the closed polydisc
$\overline{P(p,r)}$ is compactly contained in $\Omega$. There exists
a constant $A_K>0$ such that
\[
|K_\Omega(w,\eta)|\leq A_K
\]
whenever $w,\eta$ lie in polydiscs $P(p,r)$ for $p\in K$. By the Cauchy estimates,
\[
K_\Omega^{\alpha,\overline{\alpha}}(\zeta,\zeta)
\leq
A_K\frac{(\alpha!)^2}{r^{2|\alpha|}},
\qquad \zeta\in K.
\]
Using Cauchy--Schwarz, for $z,\zeta\in K$,
\begin{eqnarray*}
\left|K_\Omega^{\alpha,\ov\beta}(z,\zeta)\right|
&=&
\left| \langle 
K_{\Om}^{0,\ov\beta} (\cdot, \zeta), K_{\Om}^{0,\ov{\alpha}}(\cdot, z)
\rangle\right|
\\
&\leq&
\left(K_\Omega^{\beta,\ov{\beta}}(\zeta,\zeta)\right)^{1/2}
\left(K_\Omega^{\alpha,\ov{\alpha}}(z,z)\right)^{1/2}  \\
&\leq&
A_K
\frac{\alpha!\beta!}{r^{|\alpha|+|\beta|}}.
\end{eqnarray*}
Consequently,
\[
\left|
\frac{\xi_\alpha\overline{\xi_\beta}}{\alpha!\beta!}
K_\Omega^{\alpha,\ov\beta}(z,\zeta)
\right|
\leq
A_K|\xi_\alpha||\xi_\beta|r^{-|\alpha|-|\beta|}.
\]
Since $\xi\in \ell^1$,
\[
\sum_{\alpha,\beta\in\mathbb N^n}
|\xi_\alpha||\xi_\beta|r^{-|\alpha|-|\beta|}
=
\left(
\sum_{\alpha\in\mathbb N^n}
|\xi_\alpha|r^{-|\alpha|}
\right)^2
<\infty.
\]
Thus, by the Weierstrass $M$-test, we have proved the absolute and local uniform convergence of the series representation of $B_{\xi}$ on $\Omega\times\Omega$.

\medskip

We will now see the properties of the kernel $B_{\xi}$.

\begin{enumerate}

\item Since each term $K_\Omega^{\alpha,\ov\beta}(z,\zeta)$ is holomorphic in $z$ and
anti-holomorphic in $\zeta$, the local uniform convergence implies that
$B_\xi(z,\zeta)$ is holomorphic in $z$ and anti-holomorphic in $\zeta$.
In particular, $B_{\xi}(z,\zeta)$ is real-analytic on $\Om\times\Om$.

\item This is immediate from
\[
B_\xi(z,\zeta)
=
\langle R_\xi(\cdot,\zeta),R_\xi(\cdot,z)\rangle
=
\overline{
\langle R_\xi(\cdot,z),R_\xi(\cdot,\zeta)\rangle
}
=
\overline{B_\xi(\zeta,z)}.
\]

\item This follows from
\[
B_{\xi}(z, z) =  \langle R_{\xi}(\cdot,z), R_{\xi}(\cdot,z)\rangle 
=\|R_\xi(\cdot,z)\|_{A^2(\Omega)}^2 \geq 0.
\]

\item The Cauchy--Schwarz inequality in $A^2(\Omega)$ gives
\begin{eqnarray*}
|B_\xi(z,\zeta)|^2
&=&
|\langle R_\xi(\cdot,\zeta),R_\xi(\cdot,z)\rangle|^2 \\
&\leq&
\|R_\xi(\cdot,\zeta)\|_{A^2(\Omega)}^2
\|R_\xi(\cdot,z)\|_{A^2(\Omega)}^2 \\
&=&
B_\xi(\zeta,\zeta)B_\xi(z,z).
\end{eqnarray*}

\item Let
$\mathcal H$ be the closed subspace of $A^2(\Omega)$ generated by
$\{R_\xi(\cdot,z):z\in\Omega\}$, and let $(e_j)$ be an orthonormal basis
of $\mathcal H$. By Parseval's identity,
\[
B_\xi(z,z)
=
\|R_\xi(\cdot,z)\|_{A^2(\Omega)}^2
=
\sum_j |\langle e_j,R_\xi(\cdot,z)\rangle|^2.
\]
Set
\[
h_j(z):=\langle e_j,R_\xi(\cdot,z)\rangle.
\]
Since $R_{\xi,N}(\cdot,z)$ and $R_\xi(\cdot,z)$ are anti-holomorphic in $z$, one can use Remark \ref{rem_rxi} to see that each $h_j$ is holomorphic in $z$. For $N\geq 1$, define
\[
u_N(z)
:=
\log\left(\sum_{j=1}^{N}|h_j(z)|^2\right),
\]
where we use the convention $\log 0=-\infty$. Each $u_N$ is plurisubharmonic, and the sequence $(u_N)$ is increasing. By Parseval's identity,
\[
u_N(z)\nearrow
\log B_\xi(z,z).
\]
Furthermore, $(u_N)$ is locally bounded above. Indeed, for every
compact set $K\subset\Omega$,
\[
u_N(z)
\leq
\log B_\xi(z,z)
\leq
\log\left(\sup_{w\in K}B_\xi(w,w)\right),
\qquad z\in K,
\]
whenever the supremum is positive. If it is zero, then
$u_N\equiv-\infty$ on $K$. Since $B_\xi(z,z)$ is continuous,
$\log B_\xi(z,z)$, with the convention $\log 0=-\infty$, is upper
semicontinuous. Therefore, by the increasing-limit theorem for
plurisubharmonic functions, $\log B_\xi(z,z)$ is plurisubharmonic on $\Omega$.

\item Let $z_1,\dots,z_N\in\Omega$ and $c_1,\dots,c_N\in\mathbb C$. Then
\begin{eqnarray*}
\sum_{i,j=1}^N
\overline{c_i}c_jB_\xi(z_i,z_j)
&=&
\sum_{i,j=1}^N
\overline{c_i}c_j
\langle R_\xi(\cdot,z_j),R_\xi(\cdot,z_i)\rangle  \\
&=&
\left\|
\sum_{j=1}^N c_jR_\xi(\cdot,z_j)
\right\|_{A^2(\Omega)}^2 
\geq 0.
\end{eqnarray*}
Hence $B_\xi$ is positive-definite.
\end{enumerate}
This completes the proof.
\end{proof}

\subsection{The reproducing kernel Hilbert space of $\mathbf{B_\xi}$}

One of the principal advantages of constructing a positive-definite kernel
is that it canonically gives rise to a reproducing kernel Hilbert space.
Thus, besides realizing the $\xi$-Bergman kernel as the diagonal of the
two-variable kernel $B_\xi$, our construction yields, by the
Moore--Aronszajn theorem \cite[Theorem 2.14]{RKHS}, a unique reproducing
kernel Hilbert space $\mathcal H_\xi$ whose reproducing kernel is $B_\xi$.
More precisely, $\mathcal H_\xi$ is obtained as the completion of the complex
linear span of
\[
\{B_\xi(\cdot,z):z\in\Omega\}
\]
with respect to the canonical inner product
\[
\left\langle
\sum_{i=1}^{N}a_iB_\xi(\cdot,z_i),
\sum_{j=1}^{M}b_jB_\xi(\cdot,w_j)
\right\rangle_{\mathcal H_\xi}
:=
\sum_{i=1}^{N}
\sum_{j=1}^{M}
a_i\overline{b_j}
B_\xi(w_j,z_i),
\qquad
a_i,b_j\in\mathbb{C}.
\]
The following result shows that the abstract Hilbert space $\mathcal{H}_{\xi}$ obtained from this
construction consists entirely of holomorphic functions. For the reader's convenience, we include a proof.

\begin{thm}
Every element of $\mathcal H_\xi$ is represented by a holomorphic function on
$\Omega$. Consequently, $\mathcal H_\xi$ is a holomorphic reproducing kernel
Hilbert space on $\Omega$ with reproducing kernel $B_\xi$.
\end{thm}

\begin{proof}
For every $F \in \mathcal{H}_{\xi}$ and $z\in\Om$, we have the reproducing property 
\[
F(z) 
=
\langle F,B_\xi(\cdot,z)\rangle_{\mathcal H_\xi}
\]
Therefore,
\[
|F(z)|
=
|\langle F,B_\xi(\cdot,z)\rangle_{\mathcal H_\xi}|
\leq
\|F\|_{\mathcal H_\xi}
\|B_\xi(\cdot,z)\|_{\mathcal H_\xi}
=
\|F\|_{\mathcal H_\xi} \, B_\xi(z,z)^{1/2}.
\]
Given $F\in \mathcal H_\xi$, choose a sequence $(F_m)$ in $\operatorname{span}_{\mathbb{C}}
\{B_\xi(\cdot,z):z\in\Omega\}$ such that
\[
\Vert F_{m} - F\Vert_{\mathcal{H}_{\xi}} \rightarrow 0
\qquad \text{as }m\rightarrow\infty.
\]
Since $B_\xi(w,z)$ is holomorphic in $w$, each
$F_m$ is holomorphic on $\Omega$.
Let $K\subset \Omega$ be compact. Since $B_\xi(z,z)$ is real-analytic, it
is locally bounded. Hence, there exists $C_K>0$ such that
\[
B_\xi(z,z)^{1/2}\leq C_K
\]
for all $z\in K$. Therefore, for $z\in K$,
\[
|F_m(z)-F_\ell(z)|
\leq
\|F_m-F_\ell\|_{\mathcal H_\xi}\,B_\xi(z,z)^{1/2} \\
\leq
C_K\|F_m-F_\ell\|_{\mathcal H_\xi}.
\]
Since $(F_m)$ is Cauchy in $\mathcal H_\xi$, it follows that $(F_m)$ is
Cauchy in the space of holomorphic functions on $\Omega$. Hence, $F_m$ converges locally
uniformly to a holomorphic function on $\Omega$. We identify this locally
uniform limit with $F$. Thus, every element of $\mathcal H_\xi$ is holomorphic.
\end{proof}


\section{Realizing the $\xi$-Bergman kernel as the diagonal of $B_{\xi}$}

We first recall the definitions of the $\xi$-Bergman kernel and the off-diagonal $\xi$-Bergman kernel, introduced by Bao--Guan and Bao--Guan--Sun, respectively.

\begin{defn}[Bao--Guan, \cite{Bao-Guan_1, Bao-Guan_2}]\label{md1}
Let $\Om\subset\mathbb{C}^n$ be a domain and $\xi\in \ell^1$. For $F\in A^2(\Om)$ and $z\in\Om$, define the linear functional 
\[
(\xi\cdot F)(z)
=
\sum_{\alpha\in\mathbb N^n}
\xi_\alpha\frac{F^{(\alpha)}(z)}{\alpha!}.
\]
The $\xi$-Bergman kernel $K_{\xi}(z)$ of $\Om$ is defined by
\begin{equation}
K_{\xi}(z) := \sup_{0\neq F\in A^2(\Om)} \frac{\vert (\xi\cdot F)(z)\vert^2}{\int_{\Om}\vert F\vert^2}.
\end{equation}
\end{defn}

\begin{defn}[Bao--Guan--Sun, \cite{Bao-Guan-Sun}]\label{md2}
Let $\Om\subset\mathbb{C}^n$ be a domain and $\xi\in \ell^1$. Define
\[
m_{\xi}(z)
:=
\inf\left\{
\|F\|_{A^2(\Omega)}:
F\in A^2(\Omega),\;(\xi\cdot F)(z)=1
\right\},
\]
and let $m_{\xi}(\cdot,z)$ denote the unique function that solves the above extremal problem. It is shown that
\[
K_{\xi}(z) = \frac{1}{m_{\xi}(z)^2}.
\]
The off-diagonal $\xi$-Bergman kernel $K_{\xi}(w,z)$ is defined by
\begin{equation}
K_{\xi}(w,z) : = m_{\xi}(w,z) K_{\xi}(z).
\end{equation}
\end{defn}

\begin{thm}\label{rln}
Let $\Om\subset\mathbb{C}^n$ be a domain and $\xi\in \ell^1$. For every $\zeta\in\Om$, the kernel function $R_{\xi}(\cdot,\zeta)$ is the unique element in $A^2(\Om)$ such that
\begin{equation}
(\xi\cdot F)(\zeta) = \langle F, R_{\xi}(\cdot,\zeta)\rangle_{A^2(\Om)}
\end{equation}
for all $F\in A^2(\Om)$. Consequently, the $\xi$-Bergman kernel is given by
\begin{equation}
K_{\xi}(\zeta)
=
\|R_\xi(\cdot,\zeta)\|^2_{A^2(\Omega)}
=
B_\xi(\zeta,\zeta),
\end{equation}
while the off-diagonal $\xi$-Bergman kernel is given by
\begin{equation}
K_{\xi}(z,\zeta)
=
R_\xi(z,\zeta)
\end{equation}
for all $z,\zeta\in\Om$.
\end{thm}

\begin{proof}
For every $\z\in\Om$, the linear functional $\Gamma_{\xi,\z}: A^2(\Om) \ni F\mapsto (\xi\cdot F)(\z)\in\mathbb{C}$, is bounded; see \cite[Lemma 2.4]{Bao-Guan_1}. Using Lemma \ref{der}, we obtain
\[
(\xi\cdot F)(\zeta)
=
\sum_{\alpha\in\mathbb{N}^n}\xi_{\alpha}\frac{F^{(\alpha)}(\zeta)}{\alpha!}
=
\sum_{\alpha\in\mathbb{N}^n}\xi_{\alpha} \frac{1}{\alpha!} \langle F, K_{\Om}^{0,\overline{\alpha}}(\cdot,\zeta)\rangle
=
\left\langle F, \sum_{\alpha\in\mathbb{N}^n}\frac{\ov{\xi_{\alpha}}}{\alpha!} K_{\Om}^{0,\overline{\alpha}}(\cdot,\zeta)\right\rangle
=
\left\langle F, R_{\xi}(\cdot,\zeta)\right\rangle.
\]
Note that interchanging the infinite series and the inner product is justified; see Remark \ref{rem_rxi}. By Proposition \ref{rxi}, we have $R_{\xi}(\cdot,\zeta)\in A^2(\Om)$. Therefore, the uniqueness of $R_{\xi}$ comes from the Riesz Representation theorem. This proves the first assertion. 

\medskip

Let $\zeta\in\Om$ and assume that $R_{\xi}(\cdot,\zeta)\not\equiv 0$. Let $F\in A^2(\Om)$ be such that $(\xi\cdot F)(\zeta) =1$. Then, by the Cauchy--Schwarz inequality,
\[
1 
= 
\vert (\xi\cdot F)(\zeta) \vert
=
\vert \langle F, R_{\xi}(\cdot,\zeta)\rangle_{A^2(\Om)}\vert
\leq
\Vert F\Vert_{A^2(\Om)} \, \Vert R_{\xi}(\cdot, \zeta)\Vert_{A^2(\Om)}
\]
Thus, 
\[
\Vert F\Vert_{A^2(\Om)}
\geq
\frac{1}{\Vert R_{\xi}(\cdot, \zeta)\Vert_{A^2(\Om)}}
\]
with equality holding if and only if $F$ is a scalar multiple of $R_{\xi}(\cdot,\zeta)$. Since
\[
(\xi\cdot R_\xi(\cdot,\zeta))(\zeta)
=
\langle R_\xi(\cdot,\zeta),R_\xi(\cdot,\zeta)\rangle_{A^2(\Omega)}
=
\|R_\xi(\cdot,\zeta)\|^2_{A^2(\Omega)}
\]
and $R_{\xi}(\cdot,\zeta)\in A^2(\Om)$, we have
\begin{equation}
m_{\xi}(\cdot,\zeta)
=
\frac{R_\xi(\cdot,\zeta)}
{\|R_\xi(\cdot,\zeta)\|^2_{A^2(\Omega)}}.
\end{equation}
Therefore,
\[
m_{\xi}(\zeta)
=
\begin{cases}
\frac{1}{\Vert R_{\xi}(\cdot,\zeta)\Vert_{A^2(\Om)}} & \text{if }R_{\xi}(\cdot,\zeta)\not\equiv 0
\\
\infty & otherwise
\end{cases}
\]
Since $K_{\xi}(\zeta) = m_{\xi}(\zeta)^{-2}$, we have proved that
\begin{equation}
K_{\xi}(\zeta)
=
\|R_\xi(\cdot,\zeta)\|^2_{A^2(\Omega)}
=
B_\xi(\zeta,\zeta).
\end{equation}
Hence, the off-diagonal $\xi$-Bergman kernel $K_{\xi}(z,\zeta) = m_{\xi}(z,\zeta) K_{\xi}(\zeta)$ is given by
\begin{equation}
K_{\xi}(z,\zeta)
=
R_\xi(z,\zeta).
\end{equation}
This finishes the proof of the theorem.
\end{proof}

\begin{rem}
\begin{enumerate}

\item The series representations of $B_{\xi}$ and $R_{\xi}$ yield, respectively, the following series representations for the $\xi$-Bergman kernel:
\begin{equation}
K_{\xi}(z)
=
\sum_{\alpha,\beta\in\mathbb{N}^n}
\frac{\xi_\alpha\overline{\xi_\beta}}{\alpha!\beta!}
K_\Omega^{\alpha,\ov\beta}(z,z),
\end{equation}
and for the off-diagonal $\xi$-Bergman kernel:
\begin{equation}
K_{\xi}(z,\zeta)
=
\sum_{\alpha\in\mathbb N^n}
\frac{\overline{\xi_\alpha}}{\alpha!}
K_{\Omega}^{0,\ov\alpha}(z,\zeta).
\end{equation}
The first series converges locally uniformly on $\Om$, and the second converges locally uniformly on $\Om\times\Om$.

\item Since $R_{\xi}(z,\zeta)$ coincides with the off-diagonal
$\xi$-Bergman kernel $K_{\xi}(z,\zeta)$, we shall use the same
notation $K_{\xi}$ for the sake of notational simplicity.

\item The kernel $B_{\xi}(z,w)$ is the unique function on $\Om\times\Om$ which is holomorphic in $z$ and antiholomorphic in $w$ such that $B_{\xi}(z,z) = K_{\xi}(z)$ for all $z\in\Om$. 

\item When $\xi=(1,0,0,\ldots)$, the functional $\xi\cdot F$ is point
evaluation. In this special case,
\[
K_\xi(z,\zeta)=K_\Omega(z,\zeta),
\]
and hence
\[
B_\xi(z,\zeta)
=
\langle K_\Omega(\cdot,\zeta),K_\Omega(\cdot,z)\rangle_{A^2(\Omega)}
=
K_\Omega(z,\zeta).
\]
\end{enumerate}
\end{rem}

\medskip

\begin{term}[The $\xi$-Bergman kernel]
In view of Theorem \ref{rln}, we henceforth call $B_{\xi,\Omega}$ the \emph{$\xi$-Bergman kernel} of $\Omega$.
Indeed, $B_{\xi,\Omega}$ is a Hermitian positive-definite
sesqui-holomorphic kernel whose restriction to the diagonal agrees with
the original $\xi$-Bergman kernel introduced by Bao and Guan:
\[
B_{\xi,\Omega}(z,z)=K_{\xi,\Omega}(z), \qquad z\in\Omega.
\]
By contrast, the off-diagonal $\xi$-Bergman kernel introduced by Bao--Guan--Sun does
not recover the $\xi$-Bergman kernel on the diagonal. Furthermore,
$B_{\xi,\Omega}$ preserves the positive-definite reproducing-kernel
structure characteristic of the classical Bergman kernel.
\end{term}

\medskip

In addition to the $\xi$-Bergman kernel and the off-diagonal
$\xi$-Bergman kernel, Bao--Guan--Sun \cite{Bao-Guan-Sun} considered
the function
\[
(\xi\cdot K_{\xi}(\cdot,\zeta))(z),
\]
and investigated their regularity. By the reproducing property of
$K_{\xi}(z,\zeta)$, one readily obtains
\begin{equation}
(\xi\cdot K_{\xi}(\cdot,\zeta))(z)
=
\langle K_{\xi}(\cdot,\zeta), K_{\xi}(\cdot,z)\rangle
=
B_\xi(z,\zeta).
\end{equation}
They proved that $K_{\xi}(z)\in C^{1,1}_{loc}(\Om)$; that is, $K_{\xi}(z)$ and its first order partial derivatives are locally Lipschitz continuous on $\Om.$ They also showed that the off-diagonal $\xi$-Bergman kernel $K_{\xi}(z,\zeta)$ and the kernel function $(\xi\cdot K_{\xi}(\cdot,\zeta))(z)$ are locally Lipschitz continuous with respect to $\zeta$, uniformly for $z,\zeta$ ranging over compact subsets of $\Om$. As an important consequence of Theorem~\ref{rln}, we strengthen their regularity results in the next theorem.

\begin{thm}\label{T: Real analyticity of xi Bergman kernel K}
Let $\Om\subset\mathbb{C}^n$ be a domain and $\xi\in\ell^1$. The following regularity results hold:
    \begin{enumerate}
        \item The $\xi$-Bergman kernel $K_{\xi}(\zeta)$ is real-analytic on $\Om$. 

        \item The off-diagonal $\xi$-Bergman kernel $K_{\xi}(z,\zeta)$ is holomorphic in $z$ and antiholomorphic in $\zeta$. In particular, $K_{\xi}(z,\zeta)$ is real-analytic on $\Om\times\Om$.

        \item The kernel function $(\xi\cdot K_{\xi}(\cdot,\zeta))(z)$ is holomorphic in $z$ and antiholomorphic in $\zeta$. In particular, $(\xi\cdot K_{\xi}(\cdot,\zeta))(z)$ is real-analytic on $\Om\times\Om$.
     \end{enumerate}
\end{thm}

\begin{proof}
    This is a direct consequence of Proposition \ref{rxi_0}, Theorem \ref{sxi}, Theorem \ref{rln}.
\end{proof}

We conclude this section with two further consequences of Theorem \ref{rln}.

\begin{cor}[Strict positive definiteness]
If $\Om$ is a bounded domain and $0\neq \xi\in \ell^1$, the $\xi$-Bergman kernel $B_{\xi}$ is strictly positive-definite. That is, if $z_1,\ldots,z_N\in \Omega$ are distinct and
$c_1,\ldots,c_N\in \mathbb C$ are not all zero, then
\[
\sum_{i,j=1}^N c_i\overline{c_j}B_\xi(z_i,z_j)>0.
\]
\end{cor}

\begin{proof}
Let $z_1,\ldots,z_N\in\Om$ be distinct. For $c_1,\ldots, c_N \in\mathbb{C}$, not all zero, we need to show that 
 \[
\left\|
\sum_{j=1}^N c_jK_\xi(\cdot,z_j)
\right\|_{A^2(\Omega)}^2 
=
\sum_{i,j=1}^N c_i\overline{c_j}B_\xi(z_i,z_j)>0.
\]
This is equivalent to showing that the functions $K_\xi(\cdot,z_1),\ldots,K_\xi(\cdot,z_N)$ are linearly independent. Suppose, on the contrary, there exist $c_1,\ldots,c_N \in\mathbb{C}$, not all zero, such that
\[
\sum_{j=1}^N \overline{c_j}\,K_\xi(\cdot,z_j)=0.
\]
By the reproducing property of the off-diagonal $\xi$-Bergman kernel, we have for every $F\in A^2(\Om)$,
\[
\sum_{j=1}^N c_j(\xi\cdot F)(z_j)=0.
\]
Since $\Om$ is bounded, for every $\lambda\in\mathbb{C}^n$, the function $F_{\lambda}(z) = \exp (\lambda\cdot z)$ lies in $A^2(\Om)$. For $\alpha\in\mathbb{N}^n$, we have $F_{\lambda}^{(\alpha)}(z) = \lambda^{\alpha} \,\exp(\lambda\cdot z)$. Therefore,
\[
0 = \sum_{j=1}^N c_j(\xi\cdot F_{\lambda})(z_j) 
=
\sum_{j=1}^N c_j \sum_{\alpha\in\mathbb{N}^n} \xi_{\alpha} \frac{\lambda^{\alpha}\,\exp(\lambda\cdot z_j)}{\alpha !}
=
\sum_{j=1}^N c_j \exp(\lambda\cdot z_j)\sum_{\alpha\in\mathbb{N}^n} \xi_{\alpha} \frac{\lambda^{\alpha}}{\alpha !}.
\]
Since $\xi\in \ell^1$, the function
\[
\Phi_{\xi}(\lambda) = \sum_{\alpha\in\mathbb{N}^n} \xi_{\alpha} \frac{\lambda^{\alpha}}{\alpha !}
\]
is entire for $\lambda\in\mathbb{C}^n$. Moreover, $\Phi_{\xi}$ is not identically zero as $\xi\neq 0$. Therefore, by the identity theorem, the entire function
\[
\sum_{j=1}^N c_j \exp(\lambda\cdot z_j) = 0
\qquad\text{for every }\lambda\in\mathbb{C}^n.
\]
Choose $\lambda_0\in\mathbb{C}^n$ not in the union of finitely many complex hyperplanes
\[
\bigcup_{i\neq j}\{\lambda\in\mathbb C^n: \lambda\cdot(z_i-z_j)=0\}.
\]
Then, $\lambda_0\cdot z_j$ are distinct for all $j$. Thus, $\exp(t\lambda_0\cdot z_j)$ are linearly independent functions such that
\[
\sum_{j=1}^N c_j \exp(t\lambda_0\cdot z_j) = 0
\qquad\text{for every } t\in\mathbb{C}.
\]
This forces $c_j = 0$ for each $j$. Hence, a contradiction. This concludes the proposition.
\end{proof}

\begin{cor}[Ramadanov-type convergence]
Let $\{\Omega_j\}_{j\geq 1}$ be an increasing sequence of domains in
$\mathbb C^n$, and let
\[
\Omega=\bigcup_{j=1}^{\infty}\Omega_j.
\]
If $\xi\in \ell^1$, then we have the convergence of the off-diagonal $\xi$-Bergman kernels
\[
K_{\xi,\Omega_j}(z,w)\longrightarrow K_{\xi,\Omega}(z,w)
\]
locally uniformly on $\Omega\times\Omega$. Likewise, we have the convergence of the $\xi$-Bergman kernels
\[
B_{\xi,\Omega_j}(z,w)\longrightarrow B_{\xi,\Omega}(z,w)
\]
locally uniformly on $\Omega\times\Omega$.
\end{cor}

\begin{proof}
By the classical Ramadanov theorem (see \cite{KrantzSCV}),
\[
K_{\Omega_j}(z,w)\longrightarrow K_\Omega(z,w)
\]
locally uniformly on $\Omega\times\Omega$. By the Cauchy
estimates, for every $\alpha,\beta\in\mathbb N^n$,
\[
K_{\Omega_j}^{\alpha,\ov\beta}(z,w)
\longrightarrow K_\Omega^{\alpha,\ov\beta}(z,w)
\]
locally uniformly on $\Omega\times\Omega$.
Let $K\subset\Omega$ be compact. Choose $r>0$ such that the polydisc $P(p,r)$ is compactly contained in $\Om$ for each $p\in K$. The set
\[
K_{r}:=\bigcup_{p\in K}\overline{P(p,r)}
\]
is a compact subset of $\Om$. There exists
$j_0\in \mathbb{N}$ such that $K_{r}\subset \Omega_j$, for all $j\ge j_0$.
Hence there exists a constant
$C_K>0$, independent of $j\ge j_0$, such that
\[
\sup_{j\ge j_0}\sup_{u,v\in K_{r}}
|K_{\Omega_j}(u,v)|\le C_K.
\]
Therefore, by the Cauchy estimates, for every
$\alpha,\beta\in\mathbb{N}^n$,
\[
\left|K_{\Omega_j}^{\alpha,\ov\beta}(z,w)\right|
\le
C_K\frac{\alpha!\beta!}{r^{|\alpha|+|\beta|}},
\qquad z,w\in K,\quad j\ge j_0.
\]
Since $\xi\in\ell^{1}$,
\[
   \sum_{\alpha\in\mathbb N^n}
       |\xi_\alpha|r^{-|\alpha|}<\infty.
\]
Thus, the series
\[
 K_{\xi,\Omega_j}(z,w)
   =\sum_{\alpha\in\mathbb N^n}
      \frac{\ov{\xi_\alpha}}{\alpha!}
      K_{\Omega_j}^{0,\ov\alpha}(z,w)
\]
and
\[
 B_{\xi,\Omega_j}(z,w)
   =\sum_{\alpha,\beta\in\mathbb N^n}
      \frac{\xi_\alpha\overline{\xi_\beta}}
           {\alpha!\beta!}
      K_{\Omega_j}^{\alpha,\ov\beta}(z,w)
\]
have tails which are uniformly small on $K\times K$, independently
of sufficiently large $j$. The convergence of each finite partial
sum therefore yields the local uniform convergence results for the off-diagonal $\xi$-Bergman kernel and the $\xi$-Bergman kernel.
\end{proof}


\section{The $\xi$-Bergman Kernel of the Upper Half-Plane}

\noindent Let $\mathbb H=\{z\in\mathbb C:\Im z>0\}$ be the upper half-plane. The Bergman kernel of $\mathbb H$ is
\begin{equation}
K(z,w)= -\frac1{\pi(z-\ov w)^2}.
\end{equation}
For $\xi\in\ell^{1}$, we compute the $\xi$-Bergman kernel $B_{\xi}(z,w)$ and the off-diagonal $\xi$-Bergman kernel $K_{\xi}(z,w)$ for $\mathbb{H}$.

\begin{thm}\label{T: Xi-Kernel for H}
    Let $\mathbb H=\{z\in\mathbb C:\Im z>0\}$ be the upper half-plane and $\xi\in \ell^{1}$. Then there exist entire functions $F_{\xi}$ and $G_{\xi}$ such that
    \begin{enumerate}
    \item the off-diagonal $\xi$-Bergman kernel is given by
    \begin{equation}\label{E: the off diagonal kernel for upper half plane}
    K_{\xi}(z,w) = \frac{-1}{\pi(z-\ov w)^2} G_{\xi}\!\left(\frac1{z-\ov w}\right).
    \end{equation}
    \item the $\xi$-Bergman kernel is given by
    \begin{equation}\label{E: the xi-kernel for upper half plane}
    B_{\xi}(z,w) = \frac{-1}{\pi (z-\ov{w})^{2}} F_{\xi}\left(\frac{1}{z-\ov{w}}\right).
    \end{equation}
    \end{enumerate}
\end{thm}

\begin{proof}
Set $u= (z-\ov w)$ to get $K(z,w)= -\frac1\pi u^{-2}$. To compute $K^{(m,\ov n)}(z,w)$, we note that $\partial_z u=1,\, \partial_{\ov w}u=-1$. Therefore, for any $p >0$,
\[
\partial_z^m(u^{-2}) = (-1)^{m}(m+1)!\,u^{-m-2} \ \ \ \text{and} \ \ \ \partial_{\ov w}^{n}(u^{-p}) = (p)(p+1) \ldots (p+n-1) u^{-p-n}.
\]
Using the above identities, we get $\partial_z^m\partial_{\ov w}^n(u^{-2}) = (-1)^{m}(m+n+1)!\,u^{-m-n-2}$.
Thus, we obtain
\begin{equation}\label{E: Mixed derivative}
K^{(m,\ov n)}(z,w) = \frac{(-1)^{m+1}(m+n+1)!}{\pi(z-\ov w)^{m+n+2}}.
\end{equation}

\begin{enumerate}

\item We recall that the off-diagonal $\xi$-Bergman kernel is given by
\[
K_{\xi}(z,w) = \sum_{n=0}^{\infty} \frac{\overline{\xi_n}}{n!} K^{(0,\overline{n})}(z,w).
\]
Using Equation \ref{E: Mixed derivative}, we get
\[
K_{\xi}(z,w) = \sum_{n=0}^{\infty} \frac{\overline{\xi_n}}{n!} \left( \frac{-1 (n+1)!}{\pi (z - \overline{w})^{n+2}}  \right) = \frac{-1}{\pi (z - \overline{w})^{2}} \sum_{n=0}^{\infty} \frac{(n+1)\overline{\xi_n}}{(z- \overline{w})^{n}},
\]
for all $z,w \in \mathbb{H}$. Define $G_{\xi} : \mathbb{C} \to \mathbb{C}$ by
\begin{equation}\label{E: the entire function G(xi)}
G_{\xi}(t) := \sum_{n=0}^{\infty} (n+1)\overline{\xi_n} t^{n}.
\end{equation}
Note that $\limsup{\abs{(n+1) \overline{\xi_{n}}}^{\frac{1}{n}}} = \limsup{\abs{\overline{\xi_{n}}}^{\frac{1}{n}}} = 0$, and so $G_{\xi}$ is indeed an entire function and satisfies
\[
K_{\xi}(z,w) = \frac{-1}{\pi (z - \overline{w})^{2}} G_{\xi}\left(\frac{1}{z- \overline{w}}\right).
\]

\item Recall that the $\xi$-Bergman kernel is given by 
\[
B_{\xi}(z,w) = \sum_{m,n\ge0} \frac{\xi_m\overline{\xi_n}} {m!n!} K^{(m,\ov n)}(z,w).
\]
Substituting the derivative formula using Equation \ref{E: Mixed derivative}, we get
\begin{align*}
B_{\xi}(z,w) 
&= \sum_{m,n\ge0} \frac{(-1)^{m+1}(m+n+1)! \xi_m\overline{\xi_n}}{\pi \,m!n! (z-\ov w)^{m+n+2}} 
\\ &= 
\frac{-1}{\pi (z - \ov{w})^{2}} \sum_{m,n \ge 0} \frac{(-1)^{m} (m+n+1)! \xi_{m} \ov{\xi_{n}}}{m! n!} \left(\frac{1}{z - \ov{w}}\right)^{m+n}.
\end{align*}
Define $F_{\xi} : \mathbb{C} \to \mathbb{C}$ by
\begin{equation}\label{E: the entire function F}
F_{\xi}(t) := \sum_{m,n \ge 0} \frac{(-1)^{m} (m+n+1)! \xi_{m} \ov{\xi_{n}}}{m! n!} t^{m+n}.
\end{equation}
It is immediate that $B_{\xi}(z,w) = \frac{-1}{\pi (z-\ov{w})^{2}} F_{\xi}\left(\frac{1}{z-\ov{w}}\right)$. Now we only need to check that $F_{\xi}$ is an entire function. Set $k = m+n$, and note that we can write
\[
F_{\xi}(t) = \sum_{k = 0}^{\infty} c_{k} t^{k},
\]
where 
\[
c_{k} = \sum_{m=0}^{k}\frac{(-1)^{m} (k+1)! \xi_m \ov{\xi_{k-m}} }{m! (k-m)!} = (k+1)\sum_{m=0}^{k}\frac{(-1)^{m} (k)! \xi_m \ov{\xi_{k-m}}}{m! (k-m)!}.
\]
Since the power series $\sum_{n=0}^{\infty} \xi_{n} t^{n}$ is an entire function, the Cauchy--Hadamard theorem tells us that $\lim_{n \to +\infty}{\abs{\xi_{n}}^{\frac{1}{n}}} = 0$. Therefore, for any $\epsilon > 0$, there exists $N \in \mathbb{N}$ such that for $n \ge N$
\[
\abs{\xi_{n}} < \epsilon^{n}.
\]
Therefore, there exists $M > 0$, such that 
\begin{equation}\label{E: Bound on xi}
\abs{\xi_{n}} < M\epsilon^{n},
\end{equation}
for all $n \in \mathbb{N}$. Now consider,
\[
\abs{c_{k}} \le (k+1)\sum_{m=0}^{k} \binom{k}{m} \abs{\xi_{k-m}} \abs{\xi_m} < M^{2} (k+1) \epsilon^{k} \sum_{m=0}^{k} \binom{k}{m} = M^2 (k+1) (2 \epsilon)^{k}
\]
This gives us $\lim_{k \to +\infty}{\abs{c_{k}}^{\frac{1}{k}}} = 0$, and thus $F_{\xi}$ is an entire function.
\end{enumerate}
\end{proof}

\no Next, we calculate these kernels explicitly for certain special cases of sequences $\xi$. First, we look at the off-diagonal $\xi$-Bergman kernel.

\begin{thm}\label{T: Explicit R_xi for H}
Let $\mathbb H=\{z\in\mathbb C:\Im z>0\}$ be the upper half-plane.

\begin{enumerate}
    \item If $\xi_{n} = \delta_{nN}$ for some $N\in\mathbb N$, then
    \[
    K_{\xi}(z,w) = -\frac{N+1}{\pi} \frac{1}{(z-\ov w)^{N+2}}.
    \]

    \item If $\xi_n=\frac{1}{n!}$, for $n \ge 0$, then $G_{\xi}(t) = (1+t)  e^{t}$, and thus
    \[
    K_{\xi}(z,w) = -\frac{1}{\pi(z-\ov w)^2} \left(1+\frac1{z-\ov w}\right) e^{\frac1{z-\ov w}}.
    \]
\end{enumerate}
\end{thm}

\begin{proof}
Recall from Equation \ref{E: the off diagonal kernel for upper half plane} in Theorem~\ref{T: Xi-Kernel for H} that
\[
K_{\xi}(z,w) = -\frac{1}{\pi(z-\overline{w})^2} G_{\xi}\!\left(\frac1{z-\overline{w}}\right),
\]
where $G_{\xi}(t) = \sum_{n=0}^{\infty} (n+1)\overline{\xi_n}t^n$.

\begin{enumerate}

\item For $\xi_{n} =\delta_{nN}$, note that $G_{\xi}(t) = (N+1)t^N$. Therefore,

\[
\begin{aligned}
K_{\xi}(z,w) &= -\frac1{\pi(z-\overline{w})^2} (N+1) \left(\frac1{z-\overline{w}}\right)^N\\ &= -\frac{N+1}{\pi} \frac1{(z-\overline{w})^{N+2}}.
\end{aligned}
\]

\no This proves the first assertion.

\item Now suppose $\xi_n=\frac1{n!}$. Note that

\[
G_{\xi}(t) = \sum_{n=0}^{\infty} \frac{n+1}{n!}t^n = \sum_{n=0}^{\infty} \frac{t^n}{n!} + \sum_{n=0}^{\infty} \frac{nt^n}{n!} = e^{t} + t e^{t} = (1+t) e^{t}.
\]

\no Substituting this into the representation formula yields

\[
K_{\xi}(z,w) = -\frac1{\pi(z-\ov w)^2} \left( 1+\frac1{z-\ov w} \right) e^{\frac1{z-\ov w}},
\]
which proves the second assertion.
\end{enumerate}
\end{proof}

\no For the classical Bergman kernel, the question of identifying the zero set is a very important and well-studied problem. In fact, we call a domain $\Omega \subset \mathbb{C}^n$ \textit{Lu Qi-Keng domain} if its Bergman kernel is zero-free. Although it is now known that a generic bounded domain of holomorphy in $\mathbb{C}^n$ is not a Lu Qi-Keng domain (see \cite{H. Boas}), nice domains like ball and polydisc do have non-vanishing Bergman kernel. Here, we study this problem for the upper half-plane $\mathbb{H}$ for the $\xi$-Bergman kernel.

\begin{defn}
    Let $\xi \in \ell^{1}$. A domain $\Omega \subset \mathbb{C}^n$ is called $\xi$-Lu Qi-Keng domain if the corresponding $\xi$-Bergman kernel of $\Omega$ is non-vanishing, that is, for all $(z,w) \in \Omega \times \Omega$
    \[
    B_{\xi, \Omega}(z,w) \neq 0.
    \]
\end{defn}

\medskip

\no Note that for $z, w \in \mathbb{H}$, the complex number $t = \frac{1}{z - \ov{w}} \in \mathbb{C} \setminus \ov{\mathbb{H}} = \{z \in \mathbb{C} : \Im{z} < 0\}$. Theorem \ref{T: Xi-Kernel for H} tells us that $\mathbb{H}$ will be a $\xi$-Lu Qi-Keng domain if and only if the corresponding entire function $F_{\xi}$ has no zeroes in the lower half plane $\mathbb{C} \setminus \ov{\mathbb{H}}$.

\begin{lem}\label{L: the S_k term}
    For $k \in \mathbb{N}$, let $S_{k} = \frac{1}{(k!)^2}\sum_{m=0}^{k} {(-1)^m \binom{k}{m}^2}$. Then
    \[
    S_{k} = \begin{cases} 
0, & \text{if } k \ \text{is odd} \\
\frac{(-1)^{j}}{(j!)^2 (2j)!}, & \text{if } k = 2j \ \text{is even} 
\end{cases}.
    \]
\end{lem}

\begin{proof}
    Let $f(x) = (1-x^{2})^{k}$. Using binomial expansion, $f(x) = \sum_{l=0}^{k} \binom{k}{l} (-1)^{l} (x)^{2l}$. Therefore, the coefficient of $x^{k}$ in $f(x)$ will be
    \begin{equation}\label{E: Coeff of xk 1}
    [x^{k}] = \begin{cases} 
0, & \text{if } k \ \text{is odd} \\
\binom{2j}{j} (-1)^{j}, & \text{if } k = 2j \ \text{is even} 
\end{cases}.
    \end{equation}

    \no Note that we can also write $f(x) = (1 - x)^{k} (1+x)^{k} = \left(\sum_{l=0}^{k} \binom{k}{l} (-1)^{l} x^{l}\right) \left(\sum_{i=0}^{k} \binom{k}{i} x^{i}\right)$. So, the coefficient of $x^{k}$ in $f(x)$ is
    \begin{equation}\label{E: Coeff of xk 2}
        [x^{k}] = \sum_{m=0}^{k} (-1)^{m} \binom{k}{m} \binom{k}{k-m} = \sum_{m=0}^{k} (-1)^{m} \binom{k}{m}^{2} = (k!)^{2} S_{k}.
    \end{equation}
    Comparing Equations \ref{E: Coeff of xk 1} and \ref{E: Coeff of xk 2}, we get
    \[
    S_{k} = \begin{cases} 
0, & \text{if } k \ \text{is odd} \\
\frac{(-1)^{j}}{(j!)^2 (2j)!}, & \text{if } k = 2j \ \text{is even} 
\end{cases}.
    \]
\end{proof}

\begin{prop}\label{P: F-xi}
    The entire function $F_{\xi}$ in Theorem \ref{T: Xi-Kernel for H} satisfies:
    \begin{enumerate}
        \item For $\xi = \{\delta_{nN}\}_{n \ge 0}$ for some $N \in \mathbb{N}$, 
        \[
        F_{\xi}(t) = \frac{(-1)^{N}(2N+1)!}{(N!)^2}t^{2N}.
        \]

        \item For $\xi = \{\frac{1}{n!}\}_{n \ge 0}$,
        \[
        F_{\xi}(t) = J_{0}(2t) + 2t J_{-1}(2t),
        \]
        where $J_{l}(z) = \sum_{m \ge 0}{\frac{(-1)^{m}}{m! (m+l)!} \left(\frac{z}{2}\right)^{2m+l}}$ for $l \in \mathbb{N}$ and $J_{l}(z) = - J_{l}(z)$ for $l\in\mathbb{Z}^-$ gives the Bessel function of the first kind.
    \end{enumerate}
\end{prop}

\begin{proof}
    (1) Recall from Equation \ref{E: the entire function F}, 
        \[
        F_{\xi}(t) = \sum_{m,n \ge 0} \frac{(-1)^{m} (m+n+1)! \xi_{n} \ov{\xi_{m}}}{m! n!} t^{m+n}.
        \]
    Since $\xi_{n} = \delta_{nN}$, we get
    \[
    F_{\xi}(t) = \frac{(-1)^{N} (N+N+1)! \xi_{N} \ov{\xi_{N}}}{N! N!} t^{2N} = \frac{(-1)^{N}(2N+1)!}{(N!)^2}t^{2N}.
    \]

\medskip
    \no (2) Again from Equation \ref{E: the entire function F} with $\xi_{n} = \frac{1}{n!}$, we get
    \[
    F_{\xi}(t) = \sum_{m,n \ge 0} \frac{(-1)^{m} (m+n+1)!}{(m!)^2 (n!)^2} t^{m+n}.
    \]

    \no Set $k = m+n$, to write $F_{\xi}(t) = \sum_{k=0}^{+\infty} c_{k} t^{k}$, where for $k \in \mathbb{N}$,
    \[
    c_{k} = (k+1)!\sum_{m=0}^{k}\frac{(-1)^{m}}{(m!)^2 ((k-m)!)^2} = \frac{(k+1)!}{(k!)^2} \sum_{m=0}^{k} {(-1)^m \binom{k}{m}^2}.
    \]
    Using Lemma \ref{L: the S_k term}, we get 
    \[
    c_{k} =  \begin{cases} 
0, & \text{if } k \ \text{is odd} \\
\frac{(-1)^{j}(2j+1)}{(j!)^2}, & \text{if } k = 2j \ \text{is even} 
\end{cases}.
    \]

    \no Therefore, we get
    \[
    F_{\xi}(t) = \sum_{j = 0}^{\infty} \frac{(-1)^{j} (2j+1)}{(j!)^{2}} t^{2j} = 2t \sum_{j=0}^{\infty} \frac{(-1)^{j}}{j! (j-1)!} t^{2j-1} + \sum_{j=0}^{\infty} \frac{(-1)^{j}}{j! j!} t^{2j} = 2t J_{-1}(2t) + J_{0}(2t).
    \]
\end{proof}

\medskip

\no To study the zero set of the $\xi$-Bergman kernel of $\mathbb{H}$, we recall some properties of the Bessel functions $J_{l}$ and general entire functions.

\begin{defn} Let $f(z) = \sum_{n=0}^{\infty} a_{n} z^{n}$ be an entire function and $\{z_{1}, z_{2}, \ldots, z_{n}, \ldots\}$ be the set of zeroes of $f$ other than origin, such that 
    \[
    r_{1} \le r_{2} \le \ldots \le r_{n} \le \ldots,
    \]
    where $r_{i} = \abs{z_{i}}$.
\begin{enumerate}
    \item (\textbf{Order of an entire function},  \cite[page 9]{R. Boas}) The order of $f$, denoted by $\rho_{f}$, is defined as 
    \[
    \rho_{f} = \limsup_{n \to \infty} \frac{n \log{n}}{\log{\left(\frac{1}{\abs{a_{n}}}\right)}},
    \]
    if it is finite.

    \item (\textbf{Genus of zero set of an entire function}, \cite[page 14]{R. Boas}) The genus of the zero set of $f$ is defined as the non-negative integer $p$, such that $p+1$ is the smallest positive integer $\alpha$ for which the series $\sum_{n=1}^{\infty}{r_{n}^{-\alpha}}$ converges.

    \item (\textbf{Canonical product of the zero set}, \cite[page 18]{R. Boas}) The canonical product of the zero set of $f$ is defined as 
    \[
    P(z) = \prod_{n=1}^{\infty}{E\left(\frac{z}{z_{n}}, p\right)},
    \]
    where $p$ is the genus of the zero set of $f$, and $E(u,p) = (1-u) \exp{\left(u + \frac{u^2}{2} + \ldots + \frac{u^{p}}{p}\right)}$.
\end{enumerate}
    
\end{defn}

\begin{thm}[Hadamard's Factorisation Theorem, {\cite[page 22]{R. Boas}}] \label{T: HFT}
    If $f(z)$ is an entire function of order $\rho$ with an $m$-fold zero at the origin, we have
    \[
    f(z) =z^{m} \exp{(Q(z))} P(z),
    \]
    where $Q(z)$ is a polynomial of degree $q \le \rho$, and $P(z)$ is the canonical product (of genus $p$) formed with the zeroes of $f$ (other than the origin).
\end{thm}

\begin{defn}[Genus of an entire function, {\cite[page 22]{R. Boas}}]
The genus of an entire function $f(z)$ is defined to be \text{max}$(p,q)$, where $p,q$ are defined in the Hadamard's factorisation of $f$.   
\end{defn}

\begin{thm}[Laguerre's Theorem, {\cite[page 23]{R. Boas}}]\label{T: Laguerre}
    If $f(z)$ is a non-constant entire function, which is real for real $z$ and has only real zeroes, and is of genus 0 or 1, then the zeroes of $f'(z)$ are also real and are separated by the zeroes of $f(z)$.
\end{thm}

\begin{prop}\label{P: Bessel Function}
    Let $f(z) = z J_{0}(2z)$, where $J_{0}$ is the Bessel function of first kind. Then
    \begin{enumerate}
        \item all zeros of $f$ are real,
        \item the order of $f$ is 1,
        \item the genus of $f$ is 1.
    \end{enumerate}
\end{prop}

\begin{proof}
    \no (1) In \cite[page 482]{Watson}, it is mentioned that all zeros of $J_{0}(z)$ are real, which in turn, tells us that all zeros of $f(z) = z J_{0}(2z)$ are real as well. Since, $J_{0}$ is an even function, the zeros of $f$ come in pairs, that is, if $z_{0}$ is a zero of $f$, then $-z_{0}$ will also be a zero. Let $\{j_{1}, j_{2}, \ldots \}$ be the set of positive zeros of $J_{0}(z)$, such that 
    $
    j_{1} \le j_{2} \le \ldots,
    $
    then the zero set of $J_{0}$ will be $\{j_{1}, -j_{1}, j_{2}, -j_{2}, \ldots\}$. Therefore, the zero set of $f(z)$ will be $\{0, \frac{j_{1}}{2}, \frac{-j_{1}}{2}, \frac{j_{2}}{2}, \frac{-j_{2}}{2}, \ldots\}$. 

\medskip
\no (2) Note that $f(z) = \sum_{j=0}^{\infty}{\frac{(-1)^{j}}{(j!)^{2}}z^{2j+1}}$, therefore
\[
a_{n} = \begin{cases} 
0, & \text{if } n \neq 2j+1 \\
\frac{(-1)^{j}}{(j!)^2}, & \text{if } n = 2j+1 
\end{cases}.
\]

\no For $n=2j+1$,
\[
\frac{n \log{n}}{\log{\left(\frac{1}{\abs{a_{n}}}\right)}} = \frac{(2j+1)\log{(2j+1)}}{2 \log{j!}}.
\]

\no Stirling's estimate tells us that $\log{j!} \sim j \log{j} -j$, which gives the order of $f$
\[
\rho_{f} = \limsup_{n \to \infty} \frac{n \log{n}}{\log{\left(\frac{1}{\abs{a_{n}}}\right)}} = \lim_{j \to \infty} \frac{(2j+1)\log{(2j+1)}}{2 \log{j!}} = 1.
\]

\medskip
\no (3) From \cite[page 502]{Watson}, we can write
    \begin{eqnarray*}
        f(z) &=& zJ_{0}(2z) = z\prod_{n=1}^{\infty}{\left(1- \frac{z^2}{\left(\frac{j_{n}}{2}\right)^2} \right)} 
        \\
        &=&
        z\prod_{n=1}^{\infty}{\left(1- \frac{z}{\left(\frac{j_{n}}{2}\right)} \right) \exp{\left(\frac{2z}{j_{n}}\right)} \left(1- \frac{z}{\left(\frac{-j_{n}}{2}\right)} \right)} \exp{\left(\frac{-2z}{j_{n}}\right)}.
    \end{eqnarray*}
From Hadamard's Factorisation Theorem \ref{T: HFT}, we get $m = 1$, $q = 0$, and $p = 1$. Thus, the genus of $f$ is max$\{p,q\} = 1$.
\end{proof}

\medskip
\no Now we are ready to study the zero set of the $\xi$-Bergman kernel of $\mathbb{H}$. 

\begin{thm}
    The upper half plane $\mathbb{H} = \{z\in\mathbb C:\Im z>0\}$ is a $\xi$-Lu Qi-Keng domain for
    \begin{enumerate}
        \item $\xi = \{\delta_{nN}\}_{n \ge 0}$, for some $N \in \mathbb{N}$.
        \item $\xi = \{\frac{1}{n!}\}_{n \ge 0}$.
    \end{enumerate}
\end{thm}

\begin{proof}
    \no As we have noted before, we only need to show that the corresponding entire function $F_{\xi}$ has no zeros in the lower half. Theorem \ref{T: Xi-Kernel for H} then implies that $\mathbb{H}$ will be $\xi$-Lu Qi-Keng domain.

\medskip
    \no (1) For $\xi = \{\delta_{nN}\}_{n \ge 0}$, Proposition \ref{P: F-xi} gives us $F_{\xi}(t) = \frac{(-1)^{N} (2N+1)!}{(N!)^{2}} t^{2N}$. It is clear that the origin is the only zero of $F_{\xi}$, and therefore $\mathbb{H}$ is a $\xi$-Lu Qi-Keng domain for $\xi = \{\delta_{nN}\}_{n \ge 0}$, for $N \in \mathbb{N}$.

    \medskip
    \no (2) We need to prove that the entire function $F_{\xi}(t) = J_{0}(2t) + 2t J_{-1}(2t)$ does not have zeros in the lower half plane $\mathbb{C} \setminus \ov{\mathbb{H}}$. We will prove something stronger, that is, we show that all the zeros of $F_{\xi}(t)$ are real. To see that, we observe that
    \[
    \frac{d}{dt}\left(t J_{0}(2t)\right) = J_{0}(2t) + t \frac{d}{dt}(J_{0}(2t)),
    \]
    and $\frac{d}{dt}(J_{0}(2t)) = \frac{d}{dt}\left(\sum_{m=0}^{\infty}{\frac{(-1)^{m}}{(m!)^{2}} t^{2m}}\right) = \left(\sum_{m=1}^{\infty}{\frac{(-1)^{m} 2m}{m! m!} t^{2m -1}}\right) = 2\left(\sum_{m=1}^{\infty}{\frac{(-1)^{m}}{m! (m-1)!} t^{2m-1}}\right) = 2 J_{-1}(2t)$. This in turn gives 
    \begin{equation}
    F_{\xi}(t) = \frac{d}{dt}\left(t J_{0}(2t)\right).
    \end{equation}

    \no Now note that $f(t) = t J_{0}(2t)$ is a non-constant entire function, and $f(t) \in \mathbb{R}$ for $t \in \mathbb{R}$ (this is due to the fact that both $t$ and $J_{0}(2t)$ satisfy this property).  The Proposition \ref{P: Bessel Function} tells us that $f(t) = t J_{0}(2t)$ has only real zeros, and is of genus $1$. Therefore, we can apply Laguerre's Theorem \ref{T: Laguerre} to $f(t) = t J_{0}(2t)$, and conclude that $F_{\xi}(t)$ has only real zeros. Thus, $\mathbb{H}$ is a $\xi$-Lu Qi-Keng domain for $\xi = \{\frac{1}{n!}\}_{n \ge 0}$.
\end{proof}


\section{Boundary Behaviour}

In this section, we study the asymptotic behaviour of the $\xi$-Bergman kernel on the diagonal $B_{\xi}(z,\ov z)$, for the upper half-plane $\mathbb{H}$, as the point $z$ approaches the boundary $\partial\mathbb H=\mathbb R$. For $z=x+iy\in\mathbb H$, $z- \overline{z} = 2iy$, and
\[
t=\frac{1}{z-\ov z}
=-\frac{i}{2y}.
\]
Thus, approaching the boundary corresponds to $t \longrightarrow -i\infty$.

\medskip

\noindent We investigate these two limits for the special choices of $\xi$ considered in the section.

\subsection{The case $\xi_n=\delta_{nN}$}

Recall that 
\[
B_{\delta_{nN}}(z,w) = (-1)^{N+1}\frac{(2N+1)!}{\pi(N!)^2} \frac{1}{(z-\ov w)^{2N+2}}
\]
Therefore, we get
\[
B_{\delta_{nN}}(z,\ov z) = (-1)^{N+1}\frac{(2N+1)!}{\pi(N!)^2} \frac{1}{(2iy)^{2N+2}} = \frac{(2N+1)!}{\pi(N!)^2 (2y)^{2N+2}}.
\]

\subsection{The case $\xi_n=\frac1{n!}$}

Recall from the previous section that

\[
B_{\xi}(z,w) = -\frac{t^2}{\pi} \left(J_{0}(2t) + 2tJ_{-1}(2t)\right) = -\frac{t^2}{\pi} \left(J_{0}(2t) - 2tJ_{1}(2t)\right),
\]
where $t = \frac1{z-\ov w}$. Along the diagonal, $t = -\frac{i}{2y}$, and therefore

\[
B_{\xi}(z,\overline{z}) = -\frac1{\pi} \left( -\frac{i}{2y} \right)^2 \left[ J_{0}\left(-\frac{i}{y}\right) + \frac{i}{y} J_{1}\left(-\frac{i}{y}\right) \right].
\]

\noindent Recall that for $\nu \in \mathbb{Z}$, the Bessel function satisfies $J_{\nu}(iz) = i^{\nu} I_{\nu}(z)$, where $I_{\nu}$ is the modified Bessel function. Therefore $J_{0}(-ix)=I_0(x), \ J_{1}(-ix)=-i I_{1}(x)$, and we obtain

\[
B_{\xi}(z,\overline{z}) = \frac{1}{4\pi y^2} \left( I_{0} \left(\frac{1}{y}\right) + \frac{1}{y} I_{1}\left(\frac{1}{y}\right) \right).
\]

\noindent Note that, as $y\to0^+$,

\[
I_{\nu}(x) \sim \frac{e^x}{\sqrt{2\pi x}}, \qquad x\to\infty.
\]

\noindent Since $\frac{1}{y}I_{1}\left(\frac{1}{y}\right)$ dominates $I_{0}\left(\frac{1}{y}\right)$, as $y \to 0^{+}$, it follows that

\[
B_{\xi}(z,\overline{z}) \sim \frac{1}{4\pi y^2} \cdot \frac{e^{1/y}}{\sqrt{2\pi}} \frac{1}{\sqrt y}.
\]

\noindent Hence, we get

\[
B_{\xi}(z,\overline{z}) \sim \frac{e^{1/\text{dist}(z, \partial \mathbb{H})}} {4\sqrt{2}\,\pi^{3/2}} \text{dist}(z, \partial \mathbb{H})^{-5/2}.
\]

\medskip

\no Thus, the kernel possesses an exponentially stronger boundary singularity than in the finitely supported case. We summarize the above discussion as follows.

\begin{thm}
For the upper half-plane $\mathbb{H}$, the $\xi$-Bergman kernel satisfies the following asymptotic behaviour.

\begin{enumerate}
\item If $\xi_n=\delta_{nN}$, then
\begin{equation}
B_{\xi}(z,\overline{z}) \sim \frac{(2N+1)!}{\pi (N!)^2} \text{dist}(z, \partial \mathbb{H})^{-2N-2}.
\end{equation}

\item If $\xi_n=\frac1{n!}$, then
\begin{equation}
B_{\xi}(z,\overline{z}) \sim \frac{e^{1/\text{dist}(z, \partial \mathbb{H})}} {4\sqrt{2}\,\pi^{3/2}} \text{dist}(z, \partial \mathbb{H})^{-5/2}.
\end{equation}
\end{enumerate}

\noindent Consequently, finite-support sequences $\xi_n = \delta_{nN}$ produce polynomial boundary singularities, whereas the infinite-support sequence $\xi_n=\frac1{n!}$ gives rise to an exponential boundary singularity.
\end{thm}


\section{Transformation Formulae}

Let $\Om_1,\Om_2\subset \mathbb{C}^n$ be domains and $F:\Om_1\rightarrow \Om_2$ be a biholomorphism. We write $J_F(z):=\det F'(z)$ for the complex Jacobian determinant of $F$. The biholomorphism $F$ induces the unitary map $U_{F}: A^2(\Om_2)\rightarrow A^2(\Om_1)$ defined by
\[
U_F(g)(z) := J_F(z) (g\circ F)(z),
\qquad
z\in \Om_1.
\]
Fix $\xi\in \ell^1$. For $\zeta\in \Om_1$, we define a linear functional $L_{\xi, \zeta}: A^2(\Om_2)\rightarrow \mathbb{C}$ by
\begin{equation}
L_{\xi, \zeta}(g) := 
(\xi\cdot U_F(g))(\zeta).
\end{equation}
Recall that $\Gamma_{\xi,\zeta}: A^2(\Om_1)\ni h\mapsto (\xi\cdot h)(\zeta)\in\mathbb{C}$ is a bounded linear functional. Since $U_F$ is unitary, we have
\[
\vert (\xi\cdot U_F(g))(\zeta) \vert 
\leq
\Vert \Gamma_{\xi,\zeta}\Vert \, \Vert U_F(g)\Vert_{A^2(\Om_1)} 
=
\Vert \Gamma_{\xi,\zeta}\Vert \, \Vert g\Vert_{A^2(\Om_2)}.
\]
So, the functional $L_{\xi,\zeta}$ is bounded. 

\begin{defn}
Given a biholomorphism $F:\Omega_{1}\to\Omega_{2}$ between domains $\Om_1, \Om_2\subset\mathbb{C}^n$, fix $\xi\in\ell^{1}$. For each $\zeta\in\Omega_{1}$, define
\begin{equation}
\left(\xi_{\zeta}^F\right)_{\alpha} := 
\sum_{\beta\in\mathbb N^n}
\frac{\xi_\beta}{\beta!}
\left.\partial_t^\beta
\Big[
J_F(\zeta+t)\,(F(\zeta+t)-F(\zeta))^\alpha
\Big]
\right|_{t=0},
\end{equation}
for all $\alpha\in\mathbb{N}^n$.
\end{defn}

\begin{prop}
We have $\xi_{\zeta}^F := (\xi_{\zeta}^F)_{\alpha} \in \ell^1$ and 
\begin{equation}
    L_{\xi,\zeta}(g) = (\xi^F_{\zeta} \cdot g)(F(\zeta))
\end{equation}
for every $g\in A^2(\Om_2)$.
\end{prop}

\begin{proof}
We will first show that $\xi^F_{\zeta}\in \ell^1$. For $t\in\mathbb{C}^n$ near $0$, set
\[
\Phi_\zeta(t):=F(\zeta+t)-F(\zeta).
\]
Fix $\rho>0$. We need to prove that
\[
\sum_{\alpha\in\mathbb{N}^{n}}
\left|\left(\xi^F_{\zeta}\right)_{\alpha}\right|\,\rho^{|\alpha|}
<\infty.
\]
Since $\Phi_{\zeta}(0)=0$, by continuity there exists $r>0$
such that the closed polydisc 
$\overline{P(0,r)}
:=
\left\{
t\in\mathbb{C}^{n}:
|t_{j}|\leq r,\ 1\leq j\leq n
\right\}$
satisfies $\zeta+\overline{P(0,r)}\Subset\Omega_{1}$
and
\[
q
:=
\max_{1\leq j\leq n}
\sup_{t\in P(0,r)}
|\Phi_{\zeta,j}(t)|
<
\frac{1}{\rho}.
\]
Also set
\[
M
:=
\sup_{t\in P(0,r)}
|J_{F}(\zeta+t)|.
\]
Then $M<\infty$. For each $\alpha\in\mathbb{N}^{n}$, define
$
h_{\alpha}(t)
:=
J_{F}(\zeta+t)\Phi_{\zeta}(t)^{\alpha}.
$
For $t\in P(0,r)$, we have
\[
|h_{\alpha}(t)|
\leq
M\prod_{j=1}^{n}
|\Phi_{\zeta,j}(t)|^{\alpha_{j}}
\leq
Mq^{|\alpha|}.
\]
Therefore, by the Cauchy estimates on the polydisc $P(0,r)$,
\[
\frac{1}{\beta!}
\left|
\partial_{t}^{\beta}h_{\alpha}(0)
\right|
\leq
Mq^{|\alpha|} \,r^{-|\beta|}
\]
for every $\alpha,\beta\in\mathbb{N}^{n}$. Consequently,
\begin{eqnarray*}
\sum_{\alpha\in\mathbb{N}^{n}}
\left|\left(\xi^F_{\zeta}\right)_{\alpha}\right|\,\rho^{|\alpha|}
&\leq&
\sum_{\alpha\in\mathbb{N}^{n}}
\sum_{\beta\in\mathbb{N}^{n}}
\frac{|\xi_{\beta}|}{\beta!}
\left|
\partial_{t}^{\beta}h_{\alpha}(0)
\right| \,\rho^{|\alpha|}
\\
&\leq&
M\sum_{\alpha\in\mathbb{N}^{n}}
(\rho q)^{|\alpha|}
\sum_{\beta\in\mathbb{N}^{n}}
|\xi_{\beta}|r^{-|\beta|}
=
M 
\prod_{j=1}^{n}
\left(
\sum_{k=0}^{\infty}(\rho q)^{k}
\right)
\sum_{\beta\in\mathbb{N}^{n}}
|\xi_{\beta}|r^{-|\beta|}
\\&=&
M
\frac{1}{(1-\rho q)^{n}}
\sum_{\beta\in\mathbb{N}^{n}}
|\xi_{\beta}|r^{-|\beta|}
<
\infty,
\end{eqnarray*}
because $\xi\in \ell^1$. Therefore, $\xi^F_{\zeta}\in \ell^1$. Let $g\in A^2(\Om_2)$. The power series expansion of $g$ around $F(\zeta)$ gives
\[
g(F(\zeta+t))
=
\sum_{\alpha\in\mathbb{N}^{n}}
\frac{g^{(\alpha)}(F(\zeta))}{\alpha!}
(F(\zeta+t)-F(\zeta))^{\alpha}
=
\sum_{\alpha\in\mathbb{N}^{n}}
\frac{g^{(\alpha)}(F(\zeta))}{\alpha!}
\Phi_{\zeta}(t)^{\alpha}.
\]
Now,
\[
L_{\xi,\zeta}(g)
=
(\xi\cdot U_{F}(g))(\zeta)
=
\sum_{\beta\in\mathbb{N}^{n}}
\frac{\xi_{\beta}}{\beta!}
\left.
\partial_{t}^{\beta}
\bigl[U_{F}(g)(\zeta+t)\bigr]
\right|_{t=0}.
\]
Since $U_{F}(g)(\zeta+t) = J_{F}(\zeta+t)\,g(F(\zeta+t))$, we have the power series expansion
\[
U_{F}(g)(\zeta+t)
=
J_{F}(\zeta+t)
\sum_{\alpha\in\mathbb{N}^{n}}
\frac{g^{(\alpha)}(F(\zeta))}{\alpha!}
\Phi_{\zeta}(t)^{\alpha}.
\]
Substituting the expression of $U_{F}(g)(\zeta+t)$ gives
\begin{eqnarray*}
L_{\xi,\zeta}(g)
&=&
\sum_{\beta\in\mathbb{N}^{n}}
\frac{\xi_{\beta}}{\beta!}
\left.
\partial_{t}^{\beta}
\left[
J_{F}(\zeta+t)
\sum_{\alpha\in\mathbb{N}^{n}}
\frac{g^{(\alpha)}(F(\zeta))}{\alpha!}
\Phi_{\zeta}(t)^{\alpha}
\right]
\right|_{t=0}
\\&=&
\sum_{\alpha\in\mathbb{N}^{n}}
\frac{g^{(\alpha)}(F(\zeta))}{\alpha!}
\sum_{\beta\in\mathbb{N}^{n}}
\frac{\xi_{\beta}}{\beta!}
\left.
\partial_{t}^{\beta}
\left[
J_{F}(\zeta+t)
\Phi_{\zeta}(t)^{\alpha}
\right]
\right|_{t=0}
\\&=&
(\xi^F_{\zeta}\cdot g)(F(\zeta)).
\end{eqnarray*}
This finishes the proof.
\end{proof}

\begin{rem}
In general, the transformed sequence $\xi_\zeta^F$ depends on the base
point $\zeta$. Indeed, let $F:\Omega_1\to\Omega_2$ be a biholomorphism
between planar domains and take
\[
\xi=(\delta_{nN})_{n\geq 0}
\]
for some $N\in\mathbb N$. Then, for $\zeta\in\Om_1$,
\[
\bigl(\xi_\zeta^F\bigr)_m
=
\frac{1}{N!}
\left.
\frac{d^N}{dt^N}
\left[
F'(\zeta+t)
\left(F(\zeta+t)-F(\zeta)\right)^m
\right]
\right|_{t=0}.
\]
Since $F(\zeta+t)-F(\zeta)$ vanishes to first order at $t=0$, it follows
that $\bigl(\xi_\zeta^F\bigr)_m=0$, $m>N$.
For the two extreme indices, we have
\[
\bigl(\xi_\zeta^F\bigr)_0
=
\frac{F^{(N+1)}(\zeta)}{N!}
\qquad
\bigl(\xi_\zeta^F\bigr)_N
=
\bigl(F'(\zeta)\bigr)^{N+1}.
\]
This shows the dependency on $\zeta\in\Om$ even though the original sequence is supported at the single index $N$.
\end{rem}

\begin{thm}
Let $F:\Omega_{1}\to\Omega_{2}$ be a biholomorphism between domains $\Om_1,\Om_2\subset \mathbb{C}^n$. 
The following transformation formulas hold.
\begin{enumerate}
\item The diagonal $\xi$-Bergman kernel satisfies
\begin{equation}
K_{\xi,\Omega_{1}}(z)
=
K_{\xi^{F}_{z},\Omega_{2}}(F(z))
\end{equation}
for every $z\in\Om_1$.

\item The off-diagonal $\xi$-Bergman kernel satisfies
\begin{equation}
K_{\xi,\Omega_{1}}(z,\zeta)
=
J_{F}(z)\,
K_{\xi^{F}_{\zeta},\Omega_{2}}
\bigl(F(z),F(\zeta)\bigr)
\end{equation}
for every $z,\zeta\in\Omega_{1}$.

\item The $\xi$-Bergman kernel $B_{\xi,\Omega_{1}}$ satisfies
\begin{equation}
B_{\xi,\Omega_{1}}(z,\zeta)
=
\left\langle
K_{\xi^{F}_{\zeta},\Omega_{2}}
\bigl(\cdot,F(\zeta)\bigr),
K_{\xi^{F}_{z},\Omega_{2}}
\bigl(\cdot,F(z)\bigr)
\right\rangle_{A^{2}(\Omega_{2})}
\end{equation}
for every $z,\zeta\in\Omega_{1}$. Equivalently,
\begin{equation}
B_{\xi,\Omega_{1}}(z,\zeta)
=
\left(
\xi^{F}_{z}\cdot
K_{\xi^{F}_{\zeta},\Omega_{2}}
\bigl(\cdot,F(\zeta)\bigr)
\right)(F(z)).
\end{equation}
\end{enumerate}
\end{thm}

\begin{proof}
Since $F$ is a biholomorphism, by the chain rule for complex Jacobians, $U_F^{-1}=U_{F^{-1}}$. Thus, the unitary map $U_F$ is surjective. Therefore, every $h\in A^2(\Omega_1)$ can be written
uniquely as $h=U_F(g)$, where $g\in A^2(\Omega_2)$, and
$\|h\|_{A^2(\Omega_1)}=\|g\|_{A^2(\Omega_2)}$. Hence
\begin{eqnarray*}
K_{\xi,\Omega_1}(z)
&=&
\sup_{0\neq h\in A^2(\Omega_1)}
\frac{|(\xi\cdot h)(z)|^2}
     {\|h\|_{A^2(\Omega_1)}^2}  \\
&=&
\sup_{0\neq g\in A^2(\Omega_2)}
\frac{|(\xi\cdot U_F(g))(z)|^2}
     {\|U_F(g)\|_{A^2(\Omega_1)}^2} \\
&=&
\sup_{0\neq g\in A^2(\Omega_2)}
\frac{|(\xi_z^F\cdot g)(F(z))|^2}
     {\|g\|_{A^2(\Omega_2)}^2} 
     =
K_{\xi_z^F,\Omega_2}(F(z)).
\end{eqnarray*}
This proves $(1)$. Fix $\zeta\in\Omega_1$. For every $g\in A^2(\Omega_2)$, the reproducing
property of the off-diagonal $\xi$-Bergman kernel gives
\begin{eqnarray*}
\left\langle
g,U_F^{-1}K_{\xi,\Omega_1}(\cdot,\zeta)
\right\rangle_{A^2(\Omega_2)}
&=&
\left\langle
U_F(g),K_{\xi,\Omega_1}(\cdot,\zeta)
\right\rangle_{A^2(\Omega_1)} \\
&=&
(\xi\cdot U_F(g))(\zeta) \\
&=&
(\xi_\zeta^F\cdot g)(F(\zeta)) 
=
\left\langle
g,K_{\xi_\zeta^F,\Omega_2}(\cdot,F(\zeta))
\right\rangle_{A^2(\Omega_2)}.
\end{eqnarray*}
Since this holds for every $g\in A^2(\Omega_2)$, the uniqueness of the
Riesz representative implies
\begin{equation}\label{aid}
U_F^{-1}K_{\xi,\Omega_1}(\cdot,\zeta)
=
K_{\xi_\zeta^F,\Omega_2}(\cdot,F(\zeta)).
\end{equation}
Applying $U_F$ to \eqref{aid} and evaluating at
$z\in\Omega_1$, we obtain
\[
K_{\xi,\Omega_1}(z,\zeta)
=
U_F\left(
K_{\xi_\zeta^F,\Omega_2}(\cdot,F(\zeta))
\right)(z) 
=
J_F(z)\,
K_{\xi_\zeta^F,\Omega_2}(F(z),F(\zeta)).
\]
This proves $(2)$.
Finally, using the relation between $B_\xi$ and the off-diagonal
$\xi$-Bergman kernel, and \eqref{aid}, we get
\begin{eqnarray*}
B_{\xi,\Omega_1}(z,\zeta)
&=&
\left\langle
K_{\xi,\Omega_1}(\cdot,\zeta),
K_{\xi,\Omega_1}(\cdot,z)
\right\rangle_{A^2(\Omega_1)} \\
&=&
\left\langle
U_F^{-1}K_{\xi,\Omega_1}(\cdot,\zeta),
U_F^{-1}K_{\xi,\Omega_1}(\cdot,z)
\right\rangle_{A^2(\Omega_2)} \\
&=&
\left\langle
K_{\xi_\zeta^F,\Omega_2}(\cdot,F(\zeta)),
K_{\xi_z^F,\Omega_2}(\cdot,F(z))
\right\rangle_{A^2(\Omega_2)}.
\end{eqnarray*}
By the reproducing property of
$K_{\xi_z^F,\Omega_2}(\cdot,F(z))$, the last expression is 
$(\xi_z^F\cdot
K_{\xi_\zeta^F,\Omega_2}(\cdot,F(\zeta))
)(F(z)).$
Therefore,
\[
B_{\xi,\Omega_1}(z,\zeta)
=
\left(
\xi_z^F\cdot
K_{\xi_\zeta^F,\Omega_2}(\cdot,F(\zeta))
\right)(F(z)).
\]
This proves $(3)$.
\end{proof}

\begin{rem}
When $F:\Om_1\rightarrow\Om_2$ is an affine map given by $F(z)=Az+b$ for some $A\in GL(n,\mathbb C)$, the sequence $\xi^F_{\zeta}$ is independent of $\zeta\in\Om_1$. We shall, therefore, use the notation $\xi^A := \xi^F_{\zeta}$. In fact, for every $\alpha\in\mathbb{N}^n$, we have
\[
\left(\xi^A\right)_{\alpha}
=
(\det A)
\sum_{\beta}
\frac{\xi_\beta}{\beta!}
\partial_t^\beta\left.[(At)^\alpha]\right\rvert_{t=0}.
\]
\end{rem}


\section{Acknowledgment}

The first named author acknowledges the support from the DST INSPIRE FACULTY research grant DST/INSPIRE/04/2024/003868, and the Research Initiation Grant IITJ/R\&D/IGRC/2025-26/53 from IIT Jodhpur. The second named author was supported by the ANRF-National Postdoctoral Fellowship, project no. PDF/2025/001158. The third named author was supported by the institute seed grant IIT/SRIC/BS/RP551/2025-2026 IIT Bhubaneswar.

\end{document}